\documentclass[a4paper,10pt]{amsart}

\usepackage[utf8]{inputenc} 

\usepackage{amsmath}
\usepackage{amstext}
\usepackage{amsfonts}
\usepackage{amsthm}
\usepackage{amssymb}
\usepackage{enumitem}

\usepackage{graphicx,color}
\usepackage{xcolor}
\usepackage{mdwlist}
\usepackage{tcolorbox}
\usepackage{fancybox}
\usepackage{framed}
\usepackage{verbatim}

\definecolor{darkblue}{rgb}{0,0,0.8}
\definecolor{darkgreen}{rgb}{0,0.8,0}

\definecolor{magenta}{rgb}{0.5,0,0.5}

\newcommand{\mathleft}{\@fleqntrue\@mathmargin0pt}

\usepackage{dsfont}
\usepackage{nicefrac}

\usepackage{hyperref}
\hypersetup{colorlinks}

\numberwithin{equation}{section}

\newcommand{\dualV}[2]{\langle #1, #2 \rangle_{V^*\times V}}
\newcommand{\dualVell}[2]{\langle #1, #2 \rangle_{V_{\ell}^*\times V_{\ell}}}

\newcommand{\dualX}[3][]{\langle #2,#3 \rangle_{#1^{*}\times #1}}
\newcommand{\inner}[3][]{( #2 , #3 )_{#1}}
\newcommand{\innerb}[3][]{\big(#2,#3\big)_{#1}}
\newcommand{\innerB}[3][]{\Big(#2,#3\Big)_{#1}}

\newcommand{\diff}[1]{\,\mathrm{d}#1}
\newcommand{\R}{\mathbb{R}}
\newcommand{\N}{\mathbb{N}}
\newcommand{\E}{\mathbb{E}}

\newcommand{\F}{\mathcal{F}}
\newcommand{\U}{\mathbf{U}}

\newcommand{\D}{\mathcal{D}}
\renewcommand{\L}{\mathcal{L}}
\renewcommand{\P}{\mathcal{P}}

\theoremstyle{plain}

\newtheorem{definition}{Definition}[section]
\newtheorem{theorem}[definition]{Theorem}
\newtheorem{lemma}[definition]{Lemma}

\newtheorem{assumption}{Assumption}

\theoremstyle{definition}
\newtheorem{remark}[definition]{Remark}

\begin{document}
	
\title[Domain decomposition for SPDEs]{A domain decomposition method for stochastic evolution equations}

\author[E.~Buckwar]{Evelyn Buckwar}
\address{Evelyn Buckwar\\
Institute of Stochastics\\
Johannes Kepler University\\
Altenberger Straße 69, 4040 Linz, Austria}
\email{evelyn.buckwar@jku.at}

\author[A.~Djurdjevac]{Ana Djurdjevac}
\address{Ana Djurdjevac\\
Freie Universt\"at Berlin\\
Arnimallee 6 \\
14195 Berlin, Germany}
\email{adjurdjevac@zedat.fu-berlin.de}

\author[M.~Eisenmann]{Monika Eisenmann}
\address{Monika Eisenmann\\
  Centre for Mathematical Sciences\\
  Lund University\\
  P.O.\ Box 118\\
  221 00 Lund, Sweden}
\email{monika.eisenmann@math.lth.se}

\begin{abstract}
    In recent years, SPDEs have become a well-studied field in mathematics. With their increase in popularity, it becomes important to efficiently approximate their solutions. Thus, our goal is a contribution towards the development of efficient and practical time-stepping methods for SPDEs.
	 
    Operator splitting schemes are a powerful tool for deterministic and stochastic differential equations. An example is given by domain decomposition schemes, where we split the domain into sub-domains. Instead of solving one expensive problem on the entire domain, we deal with cheaper problems on the sub-domains. This is particularly useful in modern computer architectures, as the sub-problems may often be solved in parallel. While splitting methods have already been used to study domain decomposition methods for deterministic PDEs, this is a new approach for SPDEs. 
    	
    We provide an abstract convergence analysis of a splitting scheme for stochastic evolution equations and state a domain decomposition scheme as an application of the setting. The theoretical results are verified through numerical experiments.
\end{abstract}


\thanks{
Part of this research took place while E. Buckwar was in Lund as Lise-Meitner Guest Professor.
A. Djurdjevac's research has been partially supported by Deutsche Forschungsgemeinschaft (DFG) through grant CRC 1114 Scaling Cascades in Complex Systems, Project Number 235221301, Project C10.
M. Eisenmann has been partially supported by Vetenskapsr\aa{}det (VR) through Grant No. 2023-03930.
The computations were enabled by resources provided by LUNARC}

\maketitle

\section{Introduction}
	
Modeling with partial differential equations (PDEs) is ubiquitous in many areas of natural sciences and engineering. Phenomena described by PDEs include models of vibrations or waves, flows of liquids or the distribution of heat. In all such models, one may take uncertainty of parameters, natural variability of evolution or extrinsic noise into account, arriving at random PDEs (RPDEs) or stochastic PDEs (SPDEs). In the former case, only parameters or initial or boundary data are described by random variables or processes, whereas in the latter case the new type of PDEs involves a driving noise, often an infinite dimensional Wiener process. Applications employing SPDEs include neuroscience \cite{lord2013,stannat2016}, micro-ferro-magnetism \cite{alouges2014,banas2014}, reservoir engineering and geo-statistical climate models \cite{geiger2012,lindstroem2011}. 
Design and analysis of numerical methods to approximate solutions of SPDEs or quantities of interest constitute an active area of research, we refer to early articles \cite{ gyoengy1998,hausenblas2002,shardlow1999}, and the book \cite{LordEtAl.2014}. However, compared to the level of theoretical and practical development of numerical methods for PDEs, this area is still in its infancy. In this article, we are in particular interested in the topic of domain decomposition, which is a standard tool for the parallelisation of numerical algorithms and thus speed-up, as well as treatment of multi-physics description systems. 
To the best of our knowledge, for SPDEs this topic is only addressed in \cite{Prohl.2012, Ji.2023, zhang2008}. While in \cite{zhang2008} the problem considered differs from our work as it is only concerned with  stationary models, the papers \cite{Prohl.2012, Ji.2023} also consider time-dependent problems. However, their underlying equations are either less general, in \cite{Prohl.2012} the stochastic heat equation is considered, or in \cite{Ji.2023} they consider the nonlinear Schr\"odinger equation which is of  a different kind.
In contrast to the very few works on domain decomposition approaches for SPDEs, there exist several articles treating RPDEs, e.g., \cite{cong2014,jin2007}.

As stated in \cite{MathewBook2008}, domain decomposition methods can be described as a class of divide-and-conquer methods for the numerical approximation of solutions of PDEs. A standard approach for a decomposition framework, called hybrid formulation in \cite{MathewBook2008}, consists of the following steps. First the domain of the PDE is decomposed into (overlapping or non-overlapping) sub-domains, then, on each sub-domain, a local PDE problem is formulated and neighbouring local problems are coupled by matching conditions. Each local problem must be well-defined and uniquely solvable, as well as consistent with the global PDE problem. Overall, the coupled system of PDEs is required to be equivalent to the original problem and the global solution can be found by summing up the local solutions using an appropriate partition of unity. Variations of this setting include Schwarz iterative methods or Lagrange Multiplier methods in the context of optimization problems. This approach works directly for time-independent PDEs, while in the case of parabolic PDEs, usually the first step is to apply an implicit time discretization method. For each time-step one obtains a time-independent problem, to which a domain decomposition method can be applied.

Another divide-and-conquer approach to solving differential equations (usually time-dependent problems) is a splitting method, we refer to \cite{GlowinskiEtAl.2016, MclachlanQuispel.2002} for overviews in the deterministic case and very exemplary to \cite{Ableidinger2016,AlamoSanzSerna.2016, Brehier.2023,Shardlow.2003} for the case of stochastic differential equations. 

Domain decomposition operator splittings, which were introduced already in 1989 by Vabishchevich \cite{Vabishchevich.1989} and called regionally additive schemes there, have been discussed recently in various contexts, see \cite{EisenmannHansen.2018,EisenmannHansen.2022, EisenmannStillfjord.2022, Geiser.2009, HansenHenningsson.2017}. The main idea, in contrast to the approach to domain decomposition described above, is to first decompose or split the differential operators and right-hand side of the PDE into a sum of component terms, where each one is defined on a sub-domain. Therefore, each sub-operator represents the vector field of the global PDE on its corresponding sub-domain. 
These split operators can then directly be incorporated in the time-integration through an operator splitting. Thus, an advantage of this formulation is that the domain decomposition and its corresponding error is directly combined in the time-integration strategy. In contrast to this, the previous approach only uses the domain decomposition as a black box method to solve elliptic equations that arise in every time step after the application of a time-stepping method. In this process of solving the arising elliptic equations, a domain decomposition method is usually used as an iterative method that can require additional computational costs due to the iterations. These additional iterations do not appear in our approach.

The purpose of this article is to extend this approach to certain class of stochastic evolution equations 
for the first time. In particular, we provide the theoretical background and analysis for a domain decomposition operator splitting applied to a stochastic evolution equation
\begin{align*}
	\begin{cases}
		\diff{u(t)}
		= \tilde{A}(t,u(t)) \diff{t} + B(t,u(t)) \diff{W(t)}, \quad t\in(0,T], \\ 
		u(0) = u_0.
	\end{cases}
\end{align*}
More precisely, we are considering an abstract Cauchy problem with nonlinear monotone and coercive drift term, additional Lipschitz nonlinear term and multiplicative $Q$-Wiener noise. The equation is studied in the variational context which is well-suited for our numerical analysis, in particular for deriving a priori bounds. In order to  obtain the final error bound for the method, we require additional regularity of the solution. More precisely, we need the solution to be H\"older continuous in time. While this cannot be guaranteed for every admissible solution in our framework, we can state one slightly more restrictive setting where the needed regularity is available in the existing literature. Note that these additional assumptions are not needed, if the H\"older regularity of the solution can be obtained by other means.
The considered semi-discrete operator splitting scheme involves a decomposition of the space into sub-spaces and splitting all the operators on these sub-spaces. On every sub-space, the obtained split Cauchy problem is discretized in time using an implicit-explicit (IMEX) Euler--Maruyama method. 

For the method, we prove well-posedness of the stated implicit equations. The path-wise well-posedness follows directly from standard PDE theory and we also provide the measurability of the obtained solution. As a particular example of our scheme, we consider a domain decomposition method. Based on the derived bound of the semi-discrete solution, we provide mean-square convergence of our numerical method.
In the last part of the paper, we illustrate the obtained convergence with numerical experiments for a two-dimensional stochastic heat equation. Here, we consider both a linear equation with a multiplicative noise term and a semi-linear equation with additive noise.

The paper is structured as follows. In Section~\ref{sec:analytical_background}, we state the stochastic evolution equation considered in this paper with all the needed assumptions. Moreover, we specify the solution concept and state a restricted setting where we can provide addtional regualrity of the solution. With this at hand we can then state the semi-discretized scheme in Section~\ref{sec:scheme}. More precisely, we give a detailed explanation about the setup of the scheme in Section~\ref{subsec:derivation_scheme}, prove the well-posedness of the implicit equation in Section~\ref{subsec:well_posedness} and conclude the section with a our main application for the splitting which is a domain decomposition of an SPDE. The main theoretical result of this paper can be found in Section~\ref{sec:error_analysis}. While we begin with some technical details needed for the error analysis in the first part of the section, we state our error bound in Section~\ref{subsec:error_bound}. This error bound is exemplified through a numerical experiment in Section~\ref{sec:numerical_experiment}. At the end of the paper, in Appendix~\ref{sec:appendix}, we collect some further auxiliary results needed in the analysis throughout this manuscript and state the details of the proof of a stated lemma.

\section{Analytical background} \label{sec:analytical_background}

In this section, we formally state the equation and gather the main assumptions concerning the setting we are interested in. 
In the following, let $T \in (0, \infty)$ be a given final time point. We then denote the underlying filtered probability space by $(\Omega, \F, \{\F_t\}_{t\in [0,T]}, \P)$, which satisfies the usual conditions, and we assume that the $Q$-Wiener process $W$ is adapted  to the filtration $\{\F_t\}_{t\in [0,T]}$. We then consider the stochastic abstract Cauchy problem:
\begin{align}\label{eq:StochEvEq}
	\begin{cases}
		\diff{u(t)}
		= \tilde{A}(t,u(t)) \diff{t} + B(t,u(t)) \diff{W(t)}, \quad t\in(0,T], \\ 
		u(0) = u_0,
	\end{cases}
\end{align}
where the operator $\tilde{A}(t,\cdot)$  consits of three different parts
\begin{align*}
	\tilde{A}(t, \cdot) = F(t, \cdot) + G(t) - A(t, \cdot)
\end{align*}
that we will explain in more detail in the following subsection. Decomposing the operator $\tilde{A}(t, \cdot)$ into these three parts provides a large degree of freedom for the numerical scheme explained in Section~\ref{subsec:derivation_scheme}. We will use the different parts of the operator to propose an implicit-explicit (IMEX) time-stepping method in combination with an operator splitting that can be used to describe a domain decomposition method.

\subsection{Assumptions} \label{subsec:assumptions} 

Let $(H, \inner[H]{\cdot}{\cdot}, \|\cdot\|_H)$ be a real, separable Hilbert space and let $(V, \|\cdot\|_V)$ be a real, separable, reflexive Banach space such that $V$ is continuously and densely embedded into $H$. In a standard way, we obtain the Gelfand triple $V \stackrel{d}{\hookrightarrow} H \cong H^* \stackrel{d}{\hookrightarrow} V^*$.
For $f \in V^*$ and $v \in V$ we denote the dual paring $\dualV{f}{v} = f(v)$, which extends the $H$-inner product.
In the following, we state the exact assumptions needed on the coefficients in \eqref{eq:StochEvEq}.

\begin{assumption} \label{ass:initialValue}
	The initial data $u_0\colon \Omega \to H$ is $\F_{0}$-measurable and belongs to $L^2(\Omega; H)$. 
\end{assumption}

\begin{assumption}\label{ass:opA}
	Let the nonlinear operator $A \colon [0,T] \times V \times \Omega \to V^*$ be progressively measurable. By writing $A(t,v)$ the map $\omega \mapsto A(t, v, \omega)$ is meant. Moreover, the following conditions hold:
	\begin{itemize}		
		\item[(i)] The operator $A(t, \cdot)$ is monotone and coercive uniformly in $t$, i.e.~there exist constants $\eta, \mu \in (0,\infty)$ such that
		\begin{align*}
			\dualV{A(t, v) - A(t, w)}{v - w} \geq \eta \|v - w\|_V^2, \quad
			\dualV{A(t, v)}{v} \geq \mu \|v \|_V^2
		\end{align*}
		almost surely for all $t\in [0,T]$, $v, w\in V$.
		\item[(ii)] The operator $A(t, \cdot)$ is bounded uniformly at $0 \in V$, $\nu_A$-H\"older continuous with respect to the first argument and Lipschitz continuous with respect to the second argument, i.e. there exist constants $\nu_A \in (0,\frac{1}{2}]$ and $K_A, L_A \in (0,\infty)$ such that
		\begin{align*}
			\|A(t, 0)\|_{V^*} \leq K_A, \quad 
			\|A(t, v) - A(\tau, w) \|_{V^*} \leq L_A \big(|t - \tau |^{\nu_A} + \|v - w \|_V\big)
		\end{align*}
		almost surely for $t, \tau \in [0,T]$ and $v, w\in V$.
	\end{itemize}
\end{assumption}

\begin{assumption}\label{ass:F}
	Let the nonlinear operator $F \colon [0,T] \times H \times \Omega \to H$ and the function $G \colon [0,T] \times \Omega \to V^*$ be progressively measurable. Again, by writing $F(t,v)$ and $G(t)$ the maps $\omega \mapsto F(t, v, \omega)$ and $\omega \mapsto G(t, \omega)$ are meant. Moreover, the operator $F(t,\cdot)$ is bounded uniformly at $0 \in H$, $\nu_F$-H\"older continuous with respect to the first variable and Lipschitz continuous with respect to the second variable, i.e. there exist constants $\nu_F \in (0,\frac{1}{2}]$ and $K_F, L_F \in [0,\infty)$ such that
	\begin{align*}
		\| F(t,0) \|_H \leq K_{F}, \quad  \|F(t,v)-F(\tau,w)\|_H \leq L_{F} \big(|t-\tau|^{\nu_F} + \|v-w\|_H\big)
	\end{align*}
	almost surely for $t, \tau\in [0,T]$ and $v, w\in H$. Moreover, there exists $\nu_G \in (0,\frac{1}{2}]$ such that the function $G$ is an element of $C^{\nu_G} ([0,T];L^2(\Omega;V^*))$.
\end{assumption}

In the following, we denote the space of linearly bounded operators from a Hilbert space $(H_1, \inner[H_1]{\cdot}{\cdot}, \|\cdot\|_{H_1})$ into a second Hilbert space $(H_2, \inner[H_2]{\cdot}{\cdot}, \|\cdot\|_{H_2})$ by $\L(H_1, H_2)$. The space is equipped with the norm $\|S\|_{\L(H_1,H_2)} = \sup_{\|v\|_{H_1} = 1} \|S v\|_{H_2}$.
Moreover, by $\mathcal{HS}(H_1,H_2)$ we denote the Hilbert space of Hilbert--Schmidt operators equipped with the norm $\|S\|_{\mathcal{HS}(H_1,H_2)} = \sqrt{ \text{tr} (S^*S)}$, where $\text{tr}$ is the trace operator.

\begin{assumption}\label{ass:W}
	Let $Q \in \L(H)$ be a self-adjoint, positive operator with finite trace whose eigenfunctions $\{\psi_{k}\}_{k \in \N}$ form  an orthonormal basis of $H$. The corresponding eigenvalues of $Q$ are denoted by $\{ q_{k}\}_{k \in \N}$. With a family $\{\beta_{k}\}_{k \in \N}$ of independent scalar Brownian motions on the probability space $(\Omega, \F, \{\F_t\}_{t \in [0,T]}, \P)$, the $Q$-Wiener noise can be stated through
	\begin{align*}
		W(t) = \sum_{k=1}^{\infty} q_{k}^{\frac{1}{2}} \psi_{k} \beta_{k}(t), 
	\end{align*}
	almost surely for $t \in [0,T]$.
\end{assumption}

\begin{assumption}\label{ass:B}
	For $L_2^0 := \mathcal{HS}(Q^{\frac{1}{2}} (H), H)$, let the $B \colon [0,T] \times H \times \Omega \to L_2^0$ be progressively measurable. Again, by writing $B(t,v)$ the map $\omega \mapsto B(t, v, \omega)$ is meant. Moreover, the operator $B(t,\cdot)$ is bounded uniformly at $0 \in H$, $\nu_B$-H\"older continuous with respect to the first variable and Lipschitz continuous with respect to the second variable, i.e. there exist constants $\nu_B \in (0,\frac{1}{2}]$, $K_B, L_B \in [0,\infty)$ such that
	\begin{align*}
		\| B(t,0) \|_{L_2^0} \leq K_{B}, \quad  \|B(t,v)-B(\tau,w)\|_{L_2^0} \leq L_{B} \big(|t-\tau|^{\nu_B} + \|v-w\|_H \big),
	\end{align*}
	almost surely for $t,\tau\in [0,T]$ and $v,w\in H$.
\end{assumption}

\subsection{Solution concept and additional regularity} \label{subsec:solution_regularity}

With the exact assumptions at hand, we can state that the stochastic evolution equation \eqref{eq:StochEvEq} has a unique variational solution. Moreover, to prove our error bounds, we also need additional H\"older regularity of the solution. While the numerical scheme is well-defined in the broader framework explained above, this setting does not guarantee the needed regularity for the error bounds. We can provide the additional regularity requirements for a subclass which we explain later in this subsection.

Let us first recall what is meant by a variational solution, as stated for example in \cite{LiuRoeckner.2015}. An $H$-valued adapted continuous process $\{u(t)\}_{t \in [0,T]}$ is a variational solution to Equation \eqref{eq:StochEvEq} on $[0,T]$ if  
\begin{align}\label{eq:exactSolVbound}
	\int_0^T \|u(t) \|^2_V \diff{t} < \infty
\end{align}
and 
\begin{equation*}
	(u(t), v)_H = (u_0, v)_H + \int_0^t \dualV{\tilde{A}(\tau, u(\tau))}{v} \diff{\tau} 
	+ \int_0^t \inner[H]{B(\tau, u(\tau))\diff{W(\tau)}}{v}  
\end{equation*}
holds almost surely for all $t \in [0,T]$ and $v \in V$.

\begin{lemma} \label{lem:existence}
	Let Assumptions~\ref{ass:initialValue}--\ref{ass:B} be fulfilled. Then the stochastic evolution equation \eqref{eq:StochEvEq} has a unique (up to $\P$-indistinguishability) variational solution that fulfills
	\begin{align}\label{eq:exactSolHbound}
		\E \Big[ \sup_{t \in [0,T]} \|u(t)\|_H^2 \Big].
	\end{align}
\end{lemma}

The proof is a standard result of well-posedness of variational solution of stochastic evolution equation, that even holds with weaker assumptions on $\tilde{A}$ and $B$, and can be found for example in \cite[Theorem~4.2.4]{LiuRoeckner.2015}.

Given the next assumption, we can provide additional regularity on the solution. This regularity is needed in our error bound in Theorem~\ref{thm:errorBound}. 

\begin{assumption}\label{ass:regularity}
	Let Assumptions~\ref{ass:initialValue}--\ref{ass:B} be fulfilled. Additionally, the following stronger conditions hold:
	\begin{itemize}
		\item[(i)]  Let  $A \colon V \to V^*$ be  a linear operator. Moreover, let $H^2 := \{ v \in V : A(v) \in H\}$ be a dense subset of $H$ such that it is a complete space with respect to the norm $\|\cdot\|_{H^2}$ and let $A \colon H^2 \subset H \to H$  be a self-adjoint and positive definite operator. The inverse of $A$ in $H$ is assumed to be compact.
		
		\item[(ii)] The initial value $u_0 \colon \Omega \to H^2$ is $\F_0$-measurable with $\E[ \|u_0\|_{H^{2}}^2] < \infty$.
		\item[(iii)] There exists $\nu_r \in [0,1)$ such that the mapping $B(t,\cdot)$ preserves additional regularity in the sense that
		\begin{align*}
			\|A^{\frac{\nu_r}{2}} B(t, v)\|_{L_{2}^0} \leq C( 1 + \|A^{\frac{\nu_r}{2}}v\|_H)
		\end{align*}
		for every $v \in H$ with $\|A^{\frac{\nu_r}{2}}v\|_H < \infty$.
	\end{itemize}
\end{assumption}

\begin{theorem} \label{thm:regularity}
	Let Assumptions~\ref{ass:initialValue}--\ref{ass:regularity}  be fulfilled. Then the solution of Equation \eqref{eq:StochEvEq} is H\"older continuous in time with values in $L^2(\Omega;V)$. More precisely, for $\nu_r \in [0,1)$ from Assumption~\ref{ass:regularity}, it follows that
	\begin{align*}
		\sup_{t, \tau \in [0,T], t \neq \tau} \frac{(\E[\| u(t) - u(\tau) \|_{V}^2])^{\frac{1}{2}}}{|t - \tau|^{\min (\frac{1}{2}, \frac{\nu_r}{2})}} < \infty.
	\end{align*}
\end{theorem}

A proof for autonomous coefficients can be found in \cite[Theorem~4.1]{KruseLarsson.2012} and for non-autonomous coefficients in \cite[Theorem~2.31]{Kruse.2014}. Note that in the cited papers mild solutions are considered. However, in our setting, results regarding equivalence of mild and variational solutions hold, see for example \cite[Section~7.5]{veraar2006stochastic}.

\section{A semi-discrete operator splitting scheme}
\label{sec:scheme}
In the following subsections, we will explain the scheme we choose for the approximation of the solution of \eqref{eq:StochEvEq}. We begin with a derivation and exact statement of the method in Section~\ref{subsec:derivation_scheme}. Here, we will use an IMEX scheme with respect to the different parts of $\tilde{A}(t, \cdot)$ and decompose both $\tilde{A}(t, \cdot)$ and $B(t, \cdot)$ further with a domain decomposition in mind. The derivation of the scheme is followed by a justification of why the appearing implicit equations are well-posed. In Section~\ref{subsec:domain_decomp}, we exemplify the splitting with our main application for the method and explain how a domain decomposition approach fits into the setting.

\subsection{Derivation of method}
\label{subsec:derivation_scheme}

In the following, we describe the strategy of the method that we will analyze in this paper to approximate the solution of \eqref{eq:StochEvEq}. The main idea behind the method can be summarized into three steps that we will explain in more detail throughout this subsection.

\begin{itemize}
	\addtolength{\itemindent}{0.5cm}
	\item[\textbf{Step 1:}] We split the data. More precisely, we decompose the space int sub-spaces and all operators into sub-operators on these sub-spaces.
	
	\item[\textbf{Step 2:}] We introduce a time discretization for every split abstract Cauchy problem. As a method, we choose an IMEX-Euler--Maruyama method. 
	
	\item[\textbf{Step 3:}] We combine the numerical solutions to the sub-problems using the Lie-splitting scheme and obtain a an approximation of the original problem. 
\end{itemize}
Combining these steps, we can define an operator splitting method for stochastic evolution equations \eqref{eq:StochEvEq}. In the following, let us explain \textbf{Step 1}--\textbf{3} in more detail.

\textbf{Step 1:}
We split the equation \eqref{eq:StochEvEq} into $s \in \N$ sub-equations. To this end, we begin by splitting the space $V$ into a sequence of sub-spaces. More precisely, for $\ell \in \{1,\dots,s\}$, let $(V_{\ell}, \|\cdot\|_{V_{\ell}})$ be real, reflexive, seperable Banach spaces, which are continuously and densely embedded into $H$, such that $\bigcap_{\ell = 1}^{s} V_{\ell} = V$ and $\sum_{\ell  = 1}^{s}\|\cdot\|_{V_s}$ is equivalent to $\|\cdot\|_V$. 
On these spaces we can now define split versions of the operators $A$, $F$, $B$ and the function $G$. 

\begin{assumption}\label{ass:Aell}
	Let the operators $A \colon [0,T] \times V \times \Omega \to V^*$, $F \colon [0, T]  \times H \times \Omega \to H$, $G \colon [0, T] \times \Omega \to V^*$ and $B \colon [0,T] \times H \times \Omega \to L_2^0$ be given as in Assumptions~\ref{ass:opA}, \ref{ass:F} and \ref{ass:B}, respectively. 
	For $\ell \in \{1,\dots,s\}$, let $A_{\ell}\colon [0,T] \times V_{\ell} \times \Omega \to  V_{\ell}^*$ also fulfill Assumption~\ref{ass:opA}, with $V$ replaced by $V_{\ell}$ and $F_{\ell} \colon [0, T]  \times H \times \Omega \to H$, $G_{\ell} \colon [0, T] \times \Omega \to V_{\ell}^*$ and $B_{\ell} \colon [0,T] \times H \times \Omega \to L_2^0$ fulfill Assumptions \ref{ass:F} and \ref{ass:B}, respectively. 
	For simplicity, assume that all the constants from Assumptions~\ref{ass:opA}, \ref{ass:F} and \ref{ass:B} are the same for both the unsplit and the split operators. 
	
	Moreover, the split operators fulfill the sum property
	\begin{align*}
		\sum_{\ell=1}^{s} A_{\ell}(t, v, \omega) &= A(t, v, \omega) \quad \text{in } V^*, \quad 
		\sum_{\ell=1}^{s} F_{\ell}(t, w, \omega) = F(t, w, \omega) \quad \text{in } H, \\
		\sum_{\ell=1}^{s} G_{\ell}(t, \omega) &= G(t, \omega) \quad \text{in } V^*, \quad 
		\sum_{\ell=1}^{s} B_{\ell}(t, w, \omega) = B(t, w, \omega) \quad \text{in } L_2^0
	\end{align*}
	for all $t \in [0,T]$, $v \in V$, $w \in H$ and almost all $\omega \in \Omega$.
\end{assumption}

Using the split data, we can define $s$ separate stochastic evolution equations of the form
\begin{align}\label{eq:SplitStochEvEq}
	\begin{cases}
		\diff{u_{\ell}(t)}
		= \tilde{A}_{\ell}(t,u_{\ell}(t)) \diff{t} + B_{\ell}(t,u_{\ell}(t)) \diff{W(t)}, \quad t\in(0,T], \\ 
		u_{\ell}(0) = u_0,
	\end{cases}
\end{align}
with $\tilde{A}_{\ell} = F_{\ell} + G_{\ell} - A_{\ell}$.

Note that in the following, we do not split the noise $W$ into $s$ parts as this is not necessary for the domain decomposition application that we have in mind. Here, every split operator only acts on a subset of the spatial domain $\D$ of the given SPDE. Thus, the split diffusion operator $B_{\ell}$ is only non-zero on one part of the spatial domain and the product $B_{\ell}(t, u(t)) \diff{W(t)}$ therefore only acts on this part of the domain even if the noise still acts on the entire domain. We allow that the noise truncation parameter $K_{\ell} \in \N$ used in \textbf{Step 2} is different on every part of the domain. This means that we may approximate the noise more or less accurately on the different domains.

\textbf{Step 2:} We propose a time-discretization of Equation \eqref{eq:StochEvEq} using the IMEX-Euler--Maruyama method, i.e. the scheme is implicit with respect to $A_{\ell}$ 
and explicit with respect to $F_{\ell}$ and $B_{\ell}$.
We define the time grid $0 = t_0 < \dots < t_N = T$ with $h_n = 
t_{n+1} - t_n$ such that $h_n \leq h$ for all $n \in \{1,\dots,N\}$, where $h \in (0,T)$ is the maximal step size.  To obtain an approximation of the Wiener process, we consider the truncation of the series representation of the form
\begin{align*}
	W^{K_{\ell}}(t) = \sum_{k=1}^{K_{\ell}} q_{k}^{\frac{1}{2}} \psi_{k} \beta_{k}(t)
\end{align*}
for $t \in [0,T]$ with $K_{\ell} \in \N$. Further, we denote the increments by $\Delta W^{n, K_{\ell}} = W^{K_{\ell}}(t_{n}) - W^{K_{\ell}}(t_{n-1})$.
We apply the IMEX-Euler--Maruyama scheme for each $\ell$-th sub-problem as 
\begin{align}\label{eq:pre_schem}
	\begin{cases}
		U_{\ell}^n - U_{\ell}^{n-1} + h_n A_{\ell}(t_n, U_{\ell}^n) = h_n F_{\ell}(t_{n-1}, U_{\ell}^{n-1}) + h_n G_{\ell}(t_n)\\
		\hspace{5.0cm}+ B_{\ell}(t_{n-1}, U^{n-1}) \Delta W^{n, K_{\ell}}
		&\text{in } 
		V_{\ell}^*,\\
		U_{\ell}^0 = u_0 &\text{in } H
	\end{cases}
\end{align}
for $n \in \{1,\dots,N\}$.

\textbf{Step 3:}
The last step is to find an approximation $U^n$ of the original problem \eqref{eq:StochEvEq}. In \eqref{eq:pre_schem}, we so far only approximated the split problems \eqref{eq:SplitStochEvEq} individually but there is no interaction between the sub-problems. This interaction we will now introduce using the Lie-splitting method. The main difference compared to \eqref{eq:pre_schem} becomes that we solve the sub-problems after one another and use the solution of one as the initial value for the next problem. This can be seen in the second line of the scheme
\begin{align}\label{eq:scheme}
	\begin{cases}
		U_{\ell}^n - U_{\ell-1}^{n} + h_n A_{\ell}(t_n, U_{\ell}^n) = h_n F_{\ell}(t_{n-1}, U_{\ell-1}^{n}) + h_n G_{\ell}(t_n)\\
		\hspace{5.0cm}+ B_{\ell}(t_{n-1}, U^{n-1}) \Delta W^{n, K_{\ell}}, 
		&\text{in } 
		V_{\ell}^*,\\
		U^n = U_s^n = U_{0}^{n+1} &\text{in } H,\\
		U_0^1 = u_0 &\text{in } H
	\end{cases}
\end{align}
for $\ell \in \{1,\dots,s\}$ and $n \in \{1,\dots,N\}$.
In \eqref{eq:scheme}, we solve $s$ sub-problems after each other within one time step and find an approximation 
\begin{align*}
	U^n = U_s^n \approx u(t_n)    
\end{align*}
for every $n \in \{1,\dots,N \}$. 

\subsection{Well-posedness of the implicit equation}
\label{subsec:well_posedness}

Our numerical scheme defined in \eqref{eq:scheme} is an implicit method in $U_{\ell}^n$. For such a method, it remains to verify that it admits a unique solution. In our case, we need to prove that there exists a unique solution $U_{\ell}^n \colon \Omega \to H$ that is $\F_{t_n}$-measurable. 

The first step is to show that for a fixed $\omega \in \Omega$, there exists a unique element $U_{\ell}^n(\omega) \in V_{\ell}$ that solves the path of \eqref{eq:scheme}. For this we can apply a standard tool from deterministic PDE theory. Since $I + h_n A_{\ell}(t_n, \cdot, \omega)$ is bounded, Lipschitz continuous, strictly monotone and coercive, we can apply a variant of the Browder--Minty theorem from \cite[Theorem~2.14]{Roubicek.2013} and obtain a unique solution $U_{\ell}^n(\omega) \in V_{\ell}$.

However, the Browder-Minty theorem does not provide any measurability results for the random variable $U_{\ell}^n$. Thus, we provide Lemma~\ref{lem:hMeasurable} below for this purpose. 
Rearranging the first equation in the scheme \eqref{eq:scheme}, its solution $U_{\ell}^n(\omega)$ is given by the root $U$ of the equation $\gamma(U, \omega)=0$, where the function $\gamma$ is defined by
\begin{align}\label{eq:DefAux_h}
	\begin{split}
		\gamma(U, \omega) & := U - U_{\ell-1}^{n}(\omega) + h_n A_{\ell}(t_n, U, \omega)
		- h_n F_{\ell}(t_{n-1}, U_{\ell-1}^{n}(\omega), \omega)\\
		&\quad - h_n G_{\ell}(t_n, \omega)
		- B_{\ell}(t_{n-1}, U^{n-1}(\omega), \omega) \Delta W^{n, K_{\ell}}(\omega).
	\end{split}
\end{align}
The result in the following lemma  ensures the measurability of the function that maps $\omega$ onto the root of $\gamma(\cdot, \omega)$.

\begin{lemma}\label{lem:hMeasurable}
	Given a probability space $(\Omega, \F, \mathcal{P})$ and a reflexive, real, seperable Banach space $\Xi$, let $\gamma \colon \Omega \times \Xi \to \Xi^*$ fulfill the following conditions, where  $\mathcal{N} \subset \Omega$ is a set with $\mathcal{P}(\mathcal{N}) = 0$:
	\begin{itemize}
		\item[(1)] The mapping $u \mapsto \dualX[\Xi]{\gamma(u, \omega)}{v}$ from $\Xi$ into $\R$ is continuous for every $v \in \Xi$ and $\omega \in \Omega \setminus \mathcal{N}$.
		\item[(2)] The mapping $\omega \mapsto \gamma(u, \omega)$ from $\Omega$ into $\Xi$ is $\F$-measurable for every $u \in \Xi$.
		\item[(3)] For every $\omega \in \Omega \setminus \mathcal{N}$, there exists a unique element $\U(\omega) \in \Xi$ such that $\gamma(\U(\omega), \omega) = 0$.
	\end{itemize}
	Consider the mapping $\U \colon \Omega \to \Xi$, $\omega 	\mapsto \U(\omega)$, where $\U(\omega)$ is the unique element in $\Xi$ described in (3) for $\omega \in \Omega \setminus \mathcal{N}$ and $\U(\omega) = 0$ for $\omega \in \mathcal{N}$. Then $\U$ is $\F$-measurable.
\end{lemma}

A proof of the previous lemma can be found in \cite{EisenmannPhD.2019}.

\begin{lemma}
	Let Assumptions~\ref{ass:initialValue}--\ref{ass:B} and Assumption~\ref{ass:Aell} be fulfilled. Then the scheme \eqref{eq:scheme} has a unique solution $U_{\ell}^n$ that is in $V_{\ell}$ almost surely, such that $\omega \mapsto U_{\ell}^n(\omega)$ is $\F_{t_n}$-measurable.
\end{lemma}

\begin{proof} 
	For $n = 1$, we begin to verify that for every $\ell \in \{1,\dots, s\}$ the mapping $\omega \mapsto U_{\ell}^1(\omega)$ is measurable. This can be done inductively over $\ell$. First, by Assumption~\ref{ass:initialValue}, the initial value $\omega \mapsto u_0(\omega) = U_0^1(\omega)$ is a $\F_{0}$-measurable function and therefore in particular $\F_{t_1}$-measurable. 
	
	Assuming that for a fixed $\ell \in \{1,\dots, s\}$, the function $\omega \mapsto U_{\ell-1}^1(\omega) \in H$ exists and is $\F_{t_1}$-measurable, the idea in the following is to apply the Browder-Minty theorem to prove the existence of $U_{\ell}^1(\omega)$ for fixed $\omega \in \Omega$. With this at hand, the measurablility of $\omega \mapsto U_{\ell}^1(\omega)$ will be shown using Lemma~\ref{lem:hMeasurable}. 
	
	In order to apply the variant of the Browder--Minty theorem, we notice that for $\omega \in \Omega$, the value $U_{\ell}^n(\omega)$ is given implicitly through the equation
	\begin{align*}
		(I + h_1 A_{\ell}(t_1, \cdot, \omega)) U_{\ell}^1(\omega) = U_{\ell-1}^{1}(\omega) + h_1 F_{\ell}(t_{0}, U_{\ell-1}^{1}(\omega), \omega) + h_1 G_{\ell}(t_1, \omega)\\
		\hspace{5cm}+ B_{\ell}(t_{0}, u_0(\omega), \omega) \Delta W^{1, K_{\ell}}(\omega)  \quad 
		&\text{in } 
		V_{1}^*.
	\end{align*}
	Since $I + h_1 A_{\ell}(t_1, \cdot, \omega)$ is bounded, Lipschitz continuous, strictly monotone and coercive, we can apply \cite[Theorem~2.14]{Roubicek.2013} to ensure the existence of a unique element $U_{\ell}^1(\omega) \in V_{\ell} \subset H$ that fulfills the equation. 
	We now consider the mapping $\gamma \colon V_{\ell} \times \Omega \to V_{\ell}^*$ as defined in \eqref{eq:DefAux_h} for $n = 1$. Due to the Browder--Minty argument above, we have found the unique element $U_{\ell}^1(\omega) \in V_{\ell}$ such that $\gamma(U_{\ell}^1(\omega), \omega) = 0$. 
	Since 
	\begin{align*}
		&\big| \dualVell{\gamma(v_1, \omega)}{w} - \dualVell{\gamma(v_2, \omega)}{w} \big|\\
		&\leq \|\gamma(v_1, \omega) - \gamma(v_2, \omega) \|_{V_{\ell}^*} \|w\|_{V_{\ell}}
		\leq \big(\|v_1 - v_2\|_{V_{\ell}^*} + h_n L_A \|v_1 - v_2 \|_{V_{\ell}}\big) \|w\|_{V_{\ell}}
	\end{align*}
	shows the mapping $v \mapsto \dualVell{\gamma(v, \omega)}{w}$ is Lipschitz continuous for every $w \in V_{\ell}$ and $\omega \in \Omega$, it is in particular continuous. 
	The mapping 
	\begin{align*}
		\omega \mapsto 
		\gamma(v, \omega)
		&= v - U_{\ell-1}^{1}(\omega) + h_1 A(t_1, v, \omega)- h_1 F(t_{0}, U_{\ell-1}^{1}(\omega), \omega)\\
		&\quad - h_1 G(t_{1},\omega) - B(t_{0}, u_0(\omega), \omega) \Delta W^{1, K_{\ell}}(\omega)
	\end{align*}
	is $\F_{t_1}$-measurable as a composition of $\F_{t_1}$-measurable functions. Thus, we can now apply Lemma~\ref{lem:hMeasurable} and find that the mapping $\omega \mapsto U_{\ell}^1(\omega)$ is $\F_{t_1}$-measurable.
	
	Altogether this shows that the function $\omega \mapsto U_0^2(\omega) = U_s^1(\omega) \in H$ exists and is $\F_{t_1}$-measurable. For the following time steps, we can argue by marching through the $n \in \{2,\dots, N\}$ and using the existence and $\F_{t_{n-1}}$-measurability of $\omega \mapsto U_0^n(\omega) = U_s^{n-1}(\omega) \in H$ as well as the same argumentation as above for the $s$ sub-steps.
\end{proof}

\subsection{Application: Domain decomposition}
\label{subsec:domain_decomp}

The abstract framework from the previous section can be applied to allow for a domain decomposition setting similar to \cite{ArrarasEtAl.2017, EisenmannHansen.2018, EisenmannStillfjord.2022, HansenHenningsson.2017} for deterministic nonlinear evolution equations. In the following, we will summarize this setting for the convenience of the reader.

To provide an example for our abstract equation \eqref{eq:StochEvEq}, we consider the setting where the operator $A$ is of Laplace type. Then, let $\D \subset \R^d$, $d \in \N$, be a bounded domain with a Lipschitz boundary $\partial \D$. We consider a stochastic semi-linear  heat equation with homogeneous Dirichlet boundary conditions
\begin{align}\label{eq:SPDE}
	\begin{cases}
		\diff{u} - \alpha(t) \Delta u \diff{t} = (\tilde{F}(t, x, u) + \tilde{G}(t,x)) \diff{t}  + \tilde{B}(t, x, u) \diff{W(t,x)}, \\  
		\hspace{8cm} t \in [0,T], x \in \D,\\
		u(t, x) = 0, \hspace{6.25cm} t \in [0,T], x \in \partial \D,\\
		u(0,x) = u_0(x), \hspace{5.55cm} x \in \D,
	\end{cases}
\end{align}
where $\alpha \colon [0,T] \times \Omega \to \R$  and $u_0 \colon \D \times \Omega \to \R$.
Moreover, the notation $\tilde{F}$, $\tilde{G}$ and $\tilde{B}$ are used to differentiate between the functions $\tilde{F} \colon [0,T] \times \D \times \R \times \Omega \to \R$ and the abstract function $F$ on $[0,T] \times L^2(\D) \times \Omega$ that it gives rise to through $[F(t, u, \omega)](x) = \tilde{F}(t,x,u,\omega)$ and analogously $\tilde{G}$ and $\tilde{B}$. 

In the following, we will use a variational formulation of \eqref{eq:SPDE} in order to introduce the domain decomposition. For this formulation, we use the concept of Sobolev spaces. The Sobolev space of all functions $v \in L^p(\D)$ such that every weak partial derivative $\partial_i v$, $ i \in \{1,\dots,d\}$, exists and is an element of $L^p(\D)$ is denoted by $W^{1,p}(\D)$. 
Additionally, for $p \in [1,\infty]$ we consider the space $W_0^{1,p}(\D)$, 
which is the closure of the infinitely many times differentiable functions with a compact support $C_c^{\infty} (\D)$ with respect to the norm of $W^{1,p}(\D)$. 
For an introduction and more details, we refer the reader to \cite{AdamsFournier.2003}.

In order to split the domain, we decompose $\D$ into $s \in \N$ overlapping sub-domains $\{ \D_{\ell} \}_{\ell =1}^{s}$. We assume that each of these subsets $\D_{\ell}$ has a Lipschitz boundary and that their union is $\D$, i.e.~$\bigcup_{\ell =1}^s \D_{\ell} = \D$.
Further, let the partition of unity $\{\chi_{\ell} \}_{\ell =1}^{s}\subset W^{1,\infty}(\D)$ be given such that for $\ell \in \{1,\dots,s\}$
\begin{align*}
	\chi_{\ell} (x)>0\text{ for all }x\in\D_{\ell},
	\quad
	\chi_{\ell} (x) = 0\text{ for all }x\in\D\setminus\D_{\ell} 
	\quad \text{and} \quad
	\sum_{\ell =1}^{s} \chi_{\ell}= 1.
\end{align*}
Together with the weight functions $\{\chi_{\ell}\}_{\ell \in \{1,\dots,s\}}$, we can now define suitable function spaces and operators. First, let the pivot space $\left(H, \inner[H]{\cdot}{\cdot}, \|\cdot\|_H \right)$ be the space $L^2(\D)$ of square integrable functions on $\D$ with the usual norm and inner product. The space $V$ is given by 
\begin{align*}
	V = \text{clos}_{\|\cdot \|_{V}} \big(C_c^{\infty}(\D)\big) = W_0^{1,2}(\D),
	\quad \text{with} \quad 
	\|\cdot\|_{V}^2
	= \|\cdot\|_H^2 + \|\nabla \cdot\|_{L^2(\D)^d}^2,
\end{align*}
where $L^2(\D)^d$ is the space of all measurable functions $v\colon \D \to \R^d$ such that the components are in $L^2(\D)$.

In order to state the spaces $V_{\ell}$, $\ell \in \{1,\dots,s\}$, we use the weighted Lebesgue space $L^2(\D_{\ell},\chi_{\ell})^d$ that consists of all measurable functions $v = (v_1,\dots,v_d) \colon \D \to \R^d$ such that
\begin{align*}
	\|(v_1,\ldots,v_{d})\|_{L^2(\D_{\ell},\chi_{\ell})^d}
	= \Big(\int_{\D_{\ell}}\chi_{\ell} |(v_1,\ldots,v_{d})|^2 
	\diff{x}\Big)^{\frac{1}{2}}
\end{align*}
is finite. Then $V_{\ell}$, $\ell \in \{1, \dots, s \}$, is given by
\begin{align*}
	V_{\ell} 
	= \text{clos}_{\|\cdot \|_{V_{\ell}}} \big(C_c^{\infty}(\D)\big),
	\quad \text{with} \quad 
	\|\cdot\|_{V_{\ell}}^2
	= \|\cdot\|_H^2 + \| \nabla \cdot \|_{L^2(\D_{\ell},\chi_{\ell})^d}^2.
\end{align*}
The spaces $V_{\ell}$, $\ell \in \{1,\dots,s\}$, are densely embedded into $H$ and their intersection $\bigcap_{\ell = 1}^s V_{\ell}$ is the original space $V$.
The operators $A(t, \cdot, \omega) \colon V \to V^*$, $F(t, \cdot, \omega) \colon H \to H$, $G(t, \omega) \in V^*$, $B(t, \cdot, \omega) \colon H \to H$, and the split operators $A_{\ell}(t, \cdot, \omega) \colon V_{\ell} \to V^*_{\ell}$, $F_{\ell}(t, \cdot, \omega) \colon H \to H$, $G_{\ell}(t, \omega) \in V_{\ell}^*$,  $B_{\ell}(t, \cdot, \omega) \colon H \to H$, $\ell\in \{ 1,\dots,s\}$, $t\in [0,T]$, are given by
\begin{align}
	\label{eq:defA}
	\begin{split}
		\dualV{A(t, v, \omega)}{w} &= \int_{\D} \alpha(t, \omega) \nabla v \cdot \nabla w \diff{x},\\ 
		\dualVell{A_{\ell}(t, v_{\ell}, \omega)}{w_{\ell}} &= \int_{\D_{\ell}} \chi_{\ell} \alpha(t, \omega) \nabla v_{\ell} \cdot  \nabla w_{\ell} \diff{x},
	\end{split}\\
	\label{eq:defG}
	\dualV{G(t, \omega)}{v} = \int_{\D} \tilde{G}(t,\cdot, \omega) &v \diff{x}, \quad \dualVell{G_{\ell}(t, \omega)}{v_{\ell}} = \int_{\D_{\ell}} \chi_{\ell}\tilde{G}(t, \cdot, \omega) v_{\ell} \diff{x}
\end{align}
for $v, w \in V$, $v_{\ell}, w_{\ell} \in V_{\ell}$ and $\omega \in \Omega$ as well as
\begin{align} 
	\label{eq:defF}
	\inner[H]{F(t, v, \omega)}{w} &= \int_{\D} \tilde{F}(t, \cdot, v, \omega) w \diff{x},
	\ \inner[H]{F_{\ell}(t, v, \omega)}{w} = \int_{\D_{\ell}} \chi_{\ell} \tilde{F}(t, \cdot, v, \omega) w \diff{x}\\
	\label{eq:defB}
	\inner[H]{B(t, v, \omega)}{w} &= \int_{\D} \tilde{B}(t, \cdot, v, \omega) v \diff{x}, \ \inner[H]{B_{\ell}(t, v, \omega)}{w} = \int_{\D_{\ell}} \chi_{\ell} \tilde{B}(t, \cdot, v, \omega) w \diff{x}
\end{align}
for $v, w \in H$ and $\omega \in \Omega$. In order to fulfill the rest of the assumptions of the previous section, we further assume that $\alpha \in C^{\nu_A}([0,T];L^{\infty}(\Omega;\R))$, $u_0 \in L^2(\D; L^2(\Omega))$, $\tilde{F}(\cdot,\cdot,v,\cdot) \in C^{\nu_F}([0,T]; L^{\infty}(\D \times \Omega;\R))$, $\tilde{G} \in C^{\nu_G}([0,T];L^2(\Omega; (W_0^{1,2}(\D))^*))$ as well as $\tilde{B}(\cdot,\cdot,v,\cdot) \in C^{\nu_B}([0,T]; L^{\infty}(\D \times \Omega;\R))$ for $v \in \R$ such that $\tilde{F}$ and $\tilde{B}$ are Lipschitz continuous in the third argument.
A proof that the setting stated in this section fulfills the Assumptions~\ref{ass:initialValue}--\ref{ass:B} and \ref{ass:Aell} can be carried through in an analogous manner to \cite[Lemma~3.1]{EisenmannStillfjord.2022} but it remains to show that the operator $B(t,v)$ is an element of $L_2^0$ almost surely. To this extend, we choose the orthonormal basis $\{\psi_k\}_{k \in \N}$ of $H$ from Assumption~\ref{ass:W}. We can then write
\begin{align*}
	\|B(t, v) \|_{L_2^0}^2 
	= \sum_{k=1}^{\infty} \|B(t, v) (q_{k}^{\frac{1}{2}} \psi_k)\|_H^2
	= \sum_{k=1}^{\infty} q_k \int_{\D} |\tilde{B}(t, \cdot, v) \psi_k|^2 \diff{x},
\end{align*}
which is bounded if $\tilde{B}(t, \cdot, v) \in L^{\infty}(\D)$ for every $v \in H$. Alternatively, if the orthonormal basis $\{\psi_k\}_{k \in \N}$ is even bounded in $L^{\infty}(\D)$, it is instead sufficient if $\tilde{B}(t, \cdot, v) \in H$ for $v \in H$ since
\begin{align*}
	\|B(t, v) \|_{L_2^0}^2 
	\leq \sum_{k=1}^{\infty} q_k \|\psi_k\|_{\infty}^2 \int_{\D} |\tilde{B}(t, \cdot, v)|^2 \diff{x}
	= \|\tilde{B}(t, \cdot, v)\|_H^2 \sum_{k=1}^{\infty} q_k \|\psi_k\|_{\infty}^2 
\end{align*}
is then bounded.

In order to obtain the additional regularity from Theorem~\ref{thm:regularity}, we need the setting to fit into Assumption~\ref{ass:regularity}. To this end, we assume additionally that $\alpha \equiv 1$, $H^2 = H^2(\D) \cap H_0^1(\D)$, $u_0 \in H^2(\D; L^2(\Omega)) \cap H_0^1(\D; L^2(\Omega))$ and the operator $B(t, \cdot)$ preserves the regularity of an element $v$ as stated in Assumption~\ref{ass:regularity}~\textit{(iii)}. This last condition is fulfilled for additive noise $B(t, v) = I_H$ (where $I_H$ is the identity operator in $H$) if the basis $\{\psi_k\}_{k \in\N}$ fulfills that they are eigenfunctions of $A$ to the corresponding eigenvalues $\{\lambda_k\}_{k \in\N}$ and $\{q_k \lambda_k^{\nu}\}_{k \in \N}$ is summable. Then  it follows that
\begin{align*}
	\|A^{\frac{\nu}{2}}B(t, v) \|_{L_2^0}^2 
	= \sum_{k=1}^{\infty} \|A^{\frac{\nu}{2}} I (q_{k}^{\frac{1}{2}} \psi_k)\|_H^2
	= \sum_{k=1}^{\infty} q_k \int_{\D} |A^{\frac{\nu}{2}}\psi_k|^2 \diff{x}
	= \sum_{k=1}^{\infty} q_k \lambda_k^{\nu}
\end{align*}
is finite. This condition is more difficult to prove for multiplicative noise.

\section{Error analysis} \label{sec:error_analysis}

In this section, we state a formal analysis for our method \eqref{eq:scheme}. This analysis contains several steps. We begin to show that the solution to our time-discrtization has bounded moments. This we will state in Lemma~\ref{lem:Apriori}. Moreover, we will present a result for the regularity of an auxiliary function that we will need for the following analysis. We conclude this section with our main result Theorem~\ref{thm:errorBound}, where we provide an explicit error bound of the scheme.

\subsection{Boundedness of the numerical solution}
\label{subsec:boundedness}

As stated in \eqref{eq:exactSolVbound} and \eqref{eq:exactSolHbound} the variational solution to the problem \eqref{eq:StochEvEq} fulfills certain moment bounds. More precisely, there exists $C_{\textup{bound, cont.}} \in (0,\infty)$ such that the exact solution $u$ fulfills
\begin{align*}
	\sup_{t \in [0,T]} \E \big[\| u(t) \|_H^2\big]	+ \int_{0}^{T} \E\big[ \|u(t)\|_V^2 \big] \diff{t}
	\leq C_{\textup{bound, cont.}}.
\end{align*}
A first step is to replicate this bound for the numerical solution. In the following lemma, we provide an a priori bound for our numerical approximation. In particular, this lemma ensures that the numerical solution fulfills the analogous moment bounds as the exact solution. 

\begin{lemma}
	\label{lem:Apriori}
	Let Assumptions~\ref{ass:initialValue}--\ref{ass:B} and Assumption~\ref{ass:Aell} be fulfilled and 
	let $U_{\ell}^n$, $\ell \in \{1,\dots,s\}$, $n \in \{1,\dots,N\}$, be the unique solution of the scheme~\eqref{eq:scheme}. Then there exists a constant 
	$C_{\textup{bound}} \in (0,\infty)$ such that for every maximal step size $h \in (0,T)$ the following a priori bound holds
	\begin{align*}
		\max_{n \in \{1,\dots,N\}} \E\big[\|U^n\|_H^2\big] + \sum_{i = 1}^{N} \sum_{\ell=1}^{s} \E\big[\| U_{\ell}^i - 	U_{\ell-1}^i\|_H^2\big]
		+ \sum_{i = 1}^{N} h_i \sum_{\ell=1}^{s} \E\big[\|U_{\ell}^i\|_{V_{\ell}}^2\big]
		\leq C_{\textup{bound}}.
	\end{align*}
\end{lemma}

The proof follows with standard arguments used for a priori bounds for evolution equation as for example in \cite[Lemma~8.6]{Roubicek.2013}. Moreover, the structure is similar to the proof of our main result in Theorem~\ref{thm:errorBound}. For the sake completeness, the proof of the Lemma~\ref{lem:Apriori} can be found in Appendix~\ref{sec:appendix_proof}.

\subsection{Regularity of the domain decomposition}
From our numerical scheme, we obtain an approximation $U^n$ of $u(t_n)$ after every time step. Every time step consists of $s$ sub-steps that take the different operators $\tilde{A}_{\ell}$ and $B_{\ell}$, $\ell \in \{1, \dots,s\}$, into account. This means that for a sub-solution $U_{\ell}^n$ we consider the operators $\tilde{A}_{j}$ and $B_{j}$ for $j$ between $1$ and $\ell$.
In the following we introduce functions $u_{\ell}$ to compare these sub-solutions to. These functions are defined as
\begin{align}\label{eq:auxUl}
	\begin{split}
		u_{\ell}(t) 
		&= u(t_{n-1}) + \sum_{j=1}^{\ell} \Big( \int_{t_{n-1}}^{t} \tilde{A}_{j}(\tau,  u(\tau)) \diff{\tau}
		+ \int_{t_{n-1}}^{t} B_{j}(\tau,  u(\tau) )  \diff{W(\tau)} \Big)\\
		&= u_{\ell-1}(t_n) + \int_{t_{n-1}}^{t} \tilde{A}_{\ell}(\tau,  u(\tau)) \diff{\tau}
		+ \int_{t_{n-1}}^{t} B_{\ell}(\tau,  u(\tau) )  \diff{W(\tau)}
	\end{split}
\end{align}
for $t \in (t_{n-1},t_n]$, $n \in \{1,\dots,N\}$ and with $\tilde{A}_{\ell} = F_{\ell} + G_{\ell} - A_{\ell}$ for $\ell \in \{1,\dots,s\}$.
In order to prove that these auxiliary functions do not differ too much from the exact solution $u$, we additionally need the following assumption.

\begin{assumption}\label{ass:MultiLevel}
	Let Assumptions~\ref{ass:opA}--\ref{ass:B} and \ref{ass:Aell} be fulfilled. There exists $C_{\textup{split}}\in (0,\infty)$ such that for all $t, \tau \in [0,T]$ it holds that
	\begin{align*}
		\begin{split}
			&\Big\| \sum_{j = 1}^{s} \Big(\int_{\tau}^{t} \tilde{A}_{j}(\sigma, v(\sigma)) \diff{\sigma}
			+ \int_{\tau}^{t} B_{j}(\sigma,  v(\sigma) )  \diff{W(\sigma)}\Big) \Big\|_{V}\\
			&\geq C_{\textup{split}} \Big\| \sum_{j = 1}^{\ell} \Big(\int_{\tau}^{t} \tilde{A}_{j}(\sigma, v(\sigma)) \diff{\sigma}
			+ \int_{\tau}^{t} B_{j}(\sigma,  v(\sigma) )  \diff{W(\sigma)}\Big) \Big\|_{V_{\ell}}\\
			&\quad + C_{\textup{split}} \Big\| \sum_{j = \ell+1}^{s} \Big(\int_{\tau}^{t} \tilde{A}_{j}(\sigma, v(\sigma)) \diff{\sigma}
			+ \int_{\tau}^{t} B_{j}(\sigma,  v(\sigma) )  \diff{W(\sigma)}\Big) \Big\|_{V_{\ell}}
		\end{split}
	\end{align*}
	almost surely for every $v \in L^2(0,T;V)$ and $\ell \in \{1,\dots,s\}$.
\end{assumption}

Note that in the case of the orthogonality of the integral operators of Lie splitting operators $\tilde{A}_{\ell}$ and $B_{\ell}$, the previous inequality becomes equality and $C_{\textup{split}}=1$. Similar types of orthogonality conditions are common in multi-grid methods. 

\begin{lemma} \label{lem:reg_uell}
	Let Assumptions~~\ref{ass:opA}--\ref{ass:B}, \ref{ass:Aell} and \ref{ass:MultiLevel} be fulfilled. Then it follows that
	\begin{align*}
		&C_{\textup{split}} \| u(t_n) - u_{\ell}(t_n) \|_{V_{\ell}} 
		+ C_{\textup{split}} \| u_{\ell}(t_n) - u(t_{n-1}) \|_{V_{\ell}} 
		\leq \| u(t_i) - u(t_{i-1}) \|_{V}.
	\end{align*}
\end{lemma}

\begin{proof}
	The difference of the function $u$ at two grid points $t_n$ and $t_{n-1}$ can be split up into two parts
	\begin{align*}
		u(t_n) - u(t_{n-1})
		&= \Big(\sum_{j = 1}^{\ell} \int_{t_{n-1}}^{t_n} \tilde{A}_{j}(\tau, u(\tau)) \diff{\tau}
		+ \sum_{j = 1}^{s} \int_{t_{n-1}}^{t_n} B_{j}(\tau,  u(\tau) )  \diff{W(\tau)}\Big)\\
		&\quad + \Big(\sum_{j = \ell + 1}^{s} \int_{t_{n-1}}^{t_n} \tilde{A}_{j}(\tau, u(\tau)) \diff{\tau}
		+ \sum_{j = 1}^{s} \int_{t_{n-1}}^{t_n} B_{j}(\tau,  u(\tau) )  \diff{W(\tau)}\Big).
	\end{align*}
	Using this decomposition and the definition of the function $u_{\ell}(t_n)$, we find that
	\begin{align*}
		u_{\ell}(t_n)
		&=u(t_n) - \sum_{j = \ell + 1}^{s} \Big(\int_{t_{n-1}}^{t_n} \tilde{A}_{j}(\tau, u(\tau)) \diff{\tau}
		+ \int_{t_{n-1}}^{t_n} B_{j}(\tau,  u(\tau) )  \diff{W(\tau)}\Big)\\
		&= u(t_{n-1}) + \sum_{j = 1}^{\ell}\Big( \int_{t_{n-1}}^{t_n} \tilde{A}_{j}(\tau, u(\tau)) \diff{\tau}
		+ \int_{t_{n-1}}^{t} B_{j}(\tau,  u(\tau) )  \diff{W(\tau)}\Big).
	\end{align*}
	This then enables us to find a representation of $u(t_n) - u_{\ell}(t_n)$ and $u_{\ell}(t_n) - u(t_{n-1})$ given by 
	\begin{align*}
		u(t_n) - u_{\ell}(t_n)
		&= \sum_{j = \ell+1}^{s} \Big(\int_{t_{n-1}}^{t_n} \tilde{A}_{j}(\tau, u(\tau)) \diff{\tau}
		+ \int_{t_{n-1}}^{t_n} B_{j}(\tau,  u(\tau) )  \diff{W(\tau)}\Big)
	\end{align*}
	and
	\begin{align*}
		u_{\ell}(t_n) - u(t_{n-1})
		&= \sum_{j = 1}^{\ell} \Big(\int_{t_{n-1}}^{t_n} \tilde{A}_{j}(\tau, u(\tau)) \diff{\tau}
		+ \int_{t_{n-1}}^{t_n} B_{j}(\tau,  u(\tau) )  \diff{W(\tau)}\Big).
	\end{align*}
	Applying Assumption~\ref{ass:MultiLevel}, it then follows that
	\begin{align*}
		&C_{\textup{split}} \| u(t_n) - u_{\ell}(t_n) \|_{V_{\ell}} + C_{\textup{split}} \| u_{\ell}(t_n) - u(t_{n-1}) \|_{V_{\ell}}\\
		&= C_{\textup{split}} \Big\| \sum_{j = 1}^{\ell} \Big(\int_{t_{n-1}}^{t_n} \tilde{A}_{j}(\tau, u(\tau)) \diff{\tau}
		+ \int_{t_{n-1}}^{t_n} B_{j}(\tau,  u(\tau) )  \diff{W(\tau)}\Big) \Big\|_{V_{\ell}}\\
		&\quad + C_{\textup{split}} \Big\| \sum_{j = \ell+1}^{s} \Big(\int_{t_{n-1}}^{t_n} \tilde{A}_{j}(\tau, u(\tau)) \diff{\tau}
		+ \int_{t_{n-1}}^{t_n} B_{j}(\tau,  u(\tau) )  \diff{W(\tau)}\Big) \Big\|_{V_{\ell}}\\
		&\leq \Big\| \sum_{j = 1}^{s} \Big(\int_{t_{n-1}}^{t_n} \tilde{A}_{j}(\tau, u(\tau)) \diff{\tau}
		+ \int_{t_{n-1}}^{t_n} B_{j}(\tau,  u(\tau) )  \diff{W(\tau)}\Big) \Big\|_{V}\\
		&= \| u(t_n) - u(t_{n-1}) \|_{V}.
	\end{align*}
\end{proof}

\subsection{Error bound}
\label{subsec:error_bound}

In this section, we provide a result on the mean-square convergence of the scheme \eqref{eq:scheme} which is a combination of an IMEX-Euler--Maruyama method and a Lie splitting scheme for the stochastic evolution equation \eqref{eq:StochEvEq}.

In the following section, we define the full error $e^n = u(t_n) - U^n$ as well as as the sub-step error $e_{\ell}^n = u_{\ell}(t_n) - U_{\ell}^n$ for $n \in \{1,\dots,N\}$ and $\ell \in \{1,\dots,s\}$. We then prove that $\E[\|e^n\|_H^2]$ tends to zero as the maximal time step size $h \to 0$. In order to prove this error bound, we summarize the necessary assumption on the setting: 
\begin{itemize}
	\item[(H1)] Let Assumptions~\ref{ass:initialValue}--\ref{ass:B} be fulfilled which ensures that the stochastic evolution equation \eqref{eq:StochEvEq} has a unique variational solution $u$.
	\item[(H2)] For $\nu_u \in (0,\frac{1}{2})$, let the exact solution $u$ of the stochastic evolution equation \eqref{eq:StochEvEq} be an element of $C^{\nu_u}([0,T]; L^2(\Omega; V))$. The H\"older exponent $\nu_u$ determines the order of the error.
	\item[(H3)] Let Assumption~\ref{ass:Aell} be fulfilled to state the decomposed problem. Further, let Assumption~\ref{ass:MultiLevel} be fulfilled to ensure that the auxiliary function $u_{\ell}$ defined in \eqref{eq:auxUl} stays sufficiently close to the solution $u$ of \eqref{eq:StochEvEq}.
	\item[(H4)] Assume that for every $\ell \in \{1,\dots,s\}$ there exists $\varepsilon_{\ell} \in (0,\infty)$ such that $\sum_{k > K_{\ell}} \E\big[ \|B_{\ell}(t,v)(q_{k}^{\frac{1}{2}} \psi_k)\|_H^2 \big] \leq \varepsilon_{\ell} \big( 1 + \E \big[ \|v\|_H^2 \big]\big)$ for every $v \in L^2(\Omega;H)$. This ensures that the truncation error $\varepsilon = \sum_{\ell = 1}^s \varepsilon_{\ell}$ is sufficiently small.
\end{itemize}

\begin{remark} \label{rem:truncatedNoise}
	The assumption $\sum_{k > K_{\ell}} \E\big[ \| B_{\ell}(t,v)(q_{k}^{\frac{1}{2}} \psi_k)\|_H^2 \big] \leq \varepsilon_{\ell} \big(1 + \E \big[ \|v\|_H^2 \big]\big)$ on the truncation is fulfilled, if the eigenvalues $\{q_k\}_{k \in \N}$ of $Q$ are summable and 
	\begin{align*}
		\E\big[ \big\| B_{\ell}(t_{i-1}, v) \big(\psi_{k}\big) \big\|_H^2 \big]
		\leq C_{\textup{truncation}} \big(1 + \E\big[ \| v \|_H^2 \big]\big)
	\end{align*}
	is fulfilled. The value $\varepsilon_{\ell}$ is then given by $\varepsilon_{\ell} = C_{\textup{truncation}} \sum_{k > K_{\ell}} q_{k}$.
	An example for this setting is if $B_{\ell}(t,v) \in \L(H)$, $t \in [0,T]$. But also weaker assumptions can be made if the eigenfunctions $\{\psi_k\}_{k \in \N}$ are more regular. 
\end{remark}

\begin{theorem}[Bound for the error] \label{thm:errorBound}
	Let $U^n$ and $U_{\ell}^n$ for $n \in \{1,\dots,N\}$ and $\ell \in \{1,\dots,s\}$ be given by the split IMEX-Euler--Maruyama method \eqref{eq:scheme}.
	Under the assumptions stated in (H1)--(H4), there exists a constant $C_{\textup{bound}} \in (0,\infty)$ such that the errors $e^n = u(t_n) - U^n$ and $e_{\ell}^n = u_{\ell}(t_n) - U_{\ell}^n$ of the scheme \eqref{eq:scheme} can be bounded by
	\begin{align*}
		&\max_{n \in \{1,\dots,N\}} \E\big[\| e^n\|_H^2\big] + \sum_{i=1}^{N} \sum_{\ell=1}^{s} \E\big[\| e_{\ell}^i - e_{\ell-1}^i\|_H^2\big] 
		+ \sum_{i=1}^{N} h_i \sum_{\ell=1}^{s} \E\big[\|e_{\ell}^i\|_{V_{\ell}}^2\big] \\
		&\leq C_{\textup{bound}} \big(h^{2\nu} + \varepsilon \big),
	\end{align*}
	where $\nu = \min \{ \nu_A, \nu_F, \nu_B, \nu_G, \nu_u \}$.
\end{theorem}

\begin{proof} 
	The proof is based on techniques from PDE analysis in a variational framework. In the following, we will state an equation for the error, which we will test with the error itself. The remaining terms can either be compensated by, for example, the monotonicty of $A_{\ell}(t_i, \cdot)$ or are small due to the regularity assumptions.
	
	We begin to subtract the numerical approximation $U_{\ell}^i$ from $u_{\ell}(t_i)$. Applying the abbreviation $e_{\ell}^i = u_{\ell}(t_i) - U_{\ell}^i$, we then obtain%
	{\allowdisplaybreaks
		\begin{align*}  
			e_{\ell}^i
			&= e_{\ell-1}^i 
			+ \int_{t_{i-1}}^{t_i} \big( F_{\ell}(t,  u(t)) - F_{\ell}(t_{i-1},  U_{\ell-1}^i)\big) \diff{t}
			+ \int_{t_{i-1}}^{t_i} \big( G_{\ell}(t) - G_{\ell}(t_{i-1}) \big) \diff{t}\\
			&\quad -\int_{t_{i-1}}^{t_i} \big( A_{\ell}(t, u(t)) - A_{\ell}(t_i, U_{\ell}^i)\big) \diff{t}\\
			&\quad + \int_{t_{i-1}}^{t_i} \big(B_{\ell}(t,  u(t)) - B_{\ell}(t_{i-1},U^{i-1}) \big)  \diff{W(t)}\\
			&\quad + B_{\ell}(t_{i-1},U^{i-1}) \big(\Delta W^{i} - \Delta W^{i, K_{\ell}} \big).
	\end{align*}}
	Next, we test the equation with $e_{\ell}^i$ and use the identity $\inner[H]{a - b}{a} = \frac{1}{2} \big( \|a\|_H^2 - \|b\|_H^2 + \|a - b\|_H^2\big)$ for $a, b \in H$. After summing up the equation from $\ell=1$ to $s$ and using the telescopic sum argument, we obtain 
	{\allowdisplaybreaks\begin{align*}
			&\frac{1}{2} \big(\|e^i\|_H^2 - \|e^{i-1}\|_H^2 + \sum_{\ell=1}^{s} \|e_{\ell}^i - e_{\ell-1}^i\|_H^2 \big)\\
			&= \int_{t_{i-1}}^{t_i} \sum_{\ell=1}^{s}\innerb[H]{F_{\ell}(t,  u(t)) - F_{\ell}(t_{i-1},  U_{\ell-1}^i)}{e_{\ell}^i} \diff{t}\\
			&\quad + \int_{t_{i-1}}^{t_i} \sum_{\ell=1}^{s}\dualVell{G_{\ell}(t) - G_{\ell}(t_{i-1}) }{e_{\ell}^i} \diff{t}\\
			&\quad - \int_{t_{i-1}}^{t_i} \sum_{\ell=1}^{s} \dualVell{A_{\ell}(t, u(t)) - A_{\ell}(t_i, U_{\ell}^i)}{e_{\ell}^i} \diff{t}\\
			&\quad + \sum_{\ell=1}^{s} \innerB[H]{\int_{t_{i-1}}^{t_i} B_{\ell}(t,  u(t)) - B_{\ell}(t_{i-1},U^{i-1}) \diff{W(t)} }{e_{\ell}^i}\\
			&\quad + \sum_{\ell=1}^{s} \innerb[H]{B_{\ell}(t_{i-1},U^{i-1}) \big(\Delta W^{i} - \Delta W^{i, K_{\ell}} \big)}{e_{\ell}^i}
			= \Gamma_1 + \Gamma_2 + \Gamma_3 + \Gamma_4 + \Gamma_5.
	\end{align*} }
	The main idea will be to split the terms such that they can be compensated by $\sum_{\ell=1}^{s} \|e_{\ell}^i - e_{\ell-1}^i\|_H^2$, the monotoicity condition on $A_{\ell}(t_i,\cdot)$ or Gr\"onwall's inequality. The remaining terms will be bounded using the regularity assumptions stated on the solution and Assumption~\ref{ass:MultiLevel}.
	First note that by assumption the exact solution $u$ is in $C^{\nu}([0,T]; L^2(\Omega; V))$. This implies, in particular, that it lies in $C^{\nu}([0,T]; L^2(\Omega; H))$ and $C^{\nu}([0,T]; L^2(\Omega; V_{\ell}))$ for every $\ell \in \{1,\dots,s\}$. Moreover, by Lemma~\ref{lem:reg_uell} and the fact that $V_{\ell}$ is continuously embedded into $H$, the distance of $u(t)$ and $u_{\ell}(t)$ is also bounded in the $V_{\ell}$ and $H$-norm. For simplicity of notation, we will use in the following that there exits a constant $C_{\textup{reg}} \in (0,\infty)$ such that
	\begin{align} \label{eq:convergence_proof_reg}
		\begin{split}
			\E\big[ \|u(t) - u(\tau)\|_H^2 \big] \leq C_{\textup{reg}} |t - \tau|^{2 \nu}, \quad
			\E\big[ \|u(t) - u(\tau)\|_{V_{\ell}}^2 \big] \leq C_{\textup{reg}} |t - \tau|^{2 \nu}, \\
			\E\big[ \|u(t) - u_{\ell}(\tau)\|_H^2 \big] \leq C_{\textup{reg}} |t - \tau|^{2 \nu}, \quad
			\E\big[ \|u(t) - u_{\ell}(\tau)\|_{V_{\ell}}^2 \big] \leq C_{\textup{reg}} |t - \tau|^{2 \nu}
		\end{split}
	\end{align}
	for all $t, \tau \in [0,T]$.
	We now bound each term $\Gamma_j$, $j \in \{1,\dots,5\}$, separately. 
	For $\Gamma_1$, we begin by adding and subtracting $e_{\ell-1}^{i}$ to the right side of the inner product. We the find that
	\begin{align*}
		\Gamma_1
		&= \int_{t_{i-1}}^{t_i} \sum_{\ell=1}^{s}\innerb[H]{F_{\ell}(t,  u(t)) - F_{\ell}(t_{i-1},  U_{\ell-1}^i)}{e_{\ell}^i - e_{\ell-1}^{i}} \diff{t}\\
		&\quad + \int_{t_{i-1}}^{t_i} \sum_{\ell=1}^{s}\innerb[H]{F_{\ell}(t,  u(t)) - F_{\ell}(t_{i-1},  U_{\ell-1}^i)}{e_{\ell-1}^{i}} \diff{t}
		= \Gamma_{1,1} + \Gamma_{1,2}.
	\end{align*}
	We will now bound $\Gamma_{1,1}$ by applying Cauchy--Schwarz inequality and Lemma~\ref{lem:CauchyInEq} with $\xi = h_i 9$ on the following three summands and inserting Assumption~\ref{ass:F}. We then obtain%
	{\allowdisplaybreaks
		\begin{align*}
			\Gamma_{1,1}
			&= \int_{t_{i-1}}^{t_i} \sum_{\ell=1}^{s}\innerb[H]{F_{\ell}(t,  u(t)) - F_{\ell}(t_{i-1},  u(t_{i-1})) }{e_{\ell}^i - e_{\ell-1}^{i}} 	\diff{t}\\
			&\quad + \int_{t_{i-1}}^{t_i} \sum_{\ell=1}^{s}\innerb[H]{F_{\ell}(t_{i-1},  u(t_{i-1})) - F_{\ell}(t_{i-1},  u_{\ell-1}(t_i) ) }{e_{\ell}^i 	- e_{\ell-1}^{i}} \diff{t}\\
			&\quad + \int_{t_{i-1}}^{t_i} \sum_{\ell=1}^{s}\innerb[H]{F_{\ell}(t,  u_{\ell-1}(t_i)) - F_{\ell}(t_{i-1},  U_{\ell-1}^i)}{e_{\ell}^i - e_{\ell-1}^{i}} \diff{t}\\
			&\leq h_i 9 L_F^2 \Big(s \int_{t_{i-1}}^{t_i} \big(|t - t_{i-1}|^{2\nu} + \| u(t) - u(t_{i-1}) \|_H^2\big) \diff{t}\\
			&\quad + h_i \sum_{\ell=1}^{s} \| u(t_{i-1}) - u_{\ell-1}(t_i) \|_H^2
			+ h_i \sum_{\ell=1}^{s} \| e_{\ell-1}^i\|_H^2\Big) + \frac{1}{12} \sum_{\ell=1}^{s} \| e_{\ell}^i - e_{\ell-1}^i\|_H^2.
	\end{align*}}
	Similarly for the next term $\Gamma_{1,2}$, we apply Lemma~\ref{lem:CauchyInEq} with $\xi = \frac{3}{4}$ and obtain
	\begin{align*}
		\Gamma_{1,2}
		&\leq \frac{3}{4} L_F^2 \Big(s \int_{t_{i-1}}^{t_i} \big(|t - t_{i-1}|^{2\nu} + \| u(t) - u(t_{i-1}) \|_H^2\big) \diff{t}\\
		&\quad + h_i \sum_{\ell=1}^{s} \|u(t_{i-1}) - u_{\ell-1}(t_i) \|_H^2\Big)
		+ h_i \Big(\frac{3}{4} L_F^2 + 1\Big) \sum_{\ell=1}^{s} \| e_{\ell-1}^i\|_H^2.
	\end{align*}
	Due to the regularity condition from \eqref{eq:convergence_proof_reg}, it then follows that
	\begin{align*}
		\E \big[\Gamma_1 \big]
		&\leq h_i^{1+ 2\nu} \Big( h \frac{9 (1 + C_{\textup{reg}})L_F^2 s }{1 + 2 \nu} + h 9 C_{\textup{reg}} L_F^2 s + \frac{3 (1 + C_{\textup{reg}})L_F^2 s }{4 + 8 \nu}
		+ \frac{3}{4} C_{\textup{reg}} L_F^2 s \Big)\\
		&\quad + \frac{1}{12} \sum_{\ell=1}^{s} \E \big[ \| e_{\ell}^i - e_{\ell-1}^i\|_H^2\big] 
		+ h_i \Big(h_i 9 L_F^2 + \frac{3}{4} L_F^2 + 1\Big) \sum_{\ell=1}^{s} \E \big[ \| e_{\ell-1}^i\|_H^2\big]\\
		&=: h_i^{1+ 2\nu} C_{\Gamma_1,1}
		+ \frac{1}{12} \sum_{\ell=1}^{s} \E \big[ \| e_{\ell}^i - e_{\ell-1}^i\|_H^2\big] 
		+ h_i C_{\Gamma_1,2} \sum_{\ell=1}^{s} \E \big[ \| e_{\ell-1}^i\|_H^2\big].
	\end{align*}
	Using the Cauchy--Schwarz inequality for the dual pairing and Lemma~\ref{lem:CauchyInEq} with $\xi = \frac{3}{2 \eta}$, where $\eta$ comes from the monotonicity condition on $A_{\ell}(t_i,\cdot)$, we find for $\E\big[\Gamma_2 \big]$ that
	\begin{align*}
		\E\big[ \Gamma_2 \big]
		&\leq \int_{t_{i-1}}^{t_i} \sum_{\ell=1}^{s} \E\big[ \| G_{\ell}(t) - G_{\ell}(t_i) \|_{V_{\ell}^*} \| e_{\ell}^i\|_{V_{\ell}}\big] \diff{t}\\
		&\leq \frac{3}{2 \eta} \int_{t_{i-1}}^{t_i} \sum_{\ell=1}^{s} \E\big[ \| G_{\ell}(t) - G_{\ell}(t_i) \|_{V_{\ell}^*}^2\big] \diff{t}
		+ h_i \frac{\eta}{6} \sum_{\ell=1}^{s} \E\big[\| e_{\ell}^i\|_{V_{\ell}}^2\big]\\
		&\leq h_i^{1 + 2 \nu} C_{\Gamma_2}
		+ h_i \frac{\eta}{6} \sum_{\ell=1}^{s} \E\big[\| e_{\ell}^i\|_{V_{\ell}}^2\big],
	\end{align*}
	since $G_{\ell} \in C^{\nu}([0,T]; L^2(\Omega; V_{\ell}^*))$.
	For the third term $\Gamma_3$, we extend it by adding and subtracting both $A_{\ell}(t_i, u(t_i))$ and $A_{\ell}(t_i, u_{\ell}(t_i))$ to the left side of the dual pairing. The resulting terms can be bounded using Cauchy--Schwarz inequality, the monotonicity condition and the Lipschitz condition from Assumption~\ref{ass:opA}, as well as Lemma~\ref{lem:CauchyInEq} with $\xi = \frac{3}{2\eta}$. It then follows that%
	{\allowdisplaybreaks
		\begin{align*}
			\E\big[\Gamma_3 \big]
			&\leq \int_{t_{i-1}}^{t_i} \sum_{\ell=1}^{s} \E\big[ \|A_{\ell}(t, u(t)) - A_{\ell}(t_i, u(t_i))\|_{V_{\ell}^*} \|e_{\ell}^i\|_{V_{\ell}}\big] \diff{t}\\
			&\quad + h_i \sum_{\ell=1}^{s} \E\big[\| A_{\ell}(t_i, u(t_i)) - A_{\ell}(t_i, u_{\ell}(t_i))\|_{V_{\ell}^*} \|e_{\ell}^i\|_{V_{\ell}} \big]
			- h_i \eta \sum_{\ell=1}^{s} \E\big[\|e_{\ell}^i\|_{V_{\ell}}^2\big] \\
			&\leq \frac{3 L_A^2 }{2 \eta} \Big( \int_{t_{i-1}}^{t_i} \sum_{\ell=1}^{s} \big(|t - t_i|^{2\nu} + \E\big[\|u(t) - u(t_i) \|_{V_{\ell}}^2\big] \big) \diff{t}\\
			&\qquad + h_i \sum_{\ell=1}^{s} \E\big[\| u(t_i) - u_{\ell}(t_i) \|_{V_{\ell}}^2\big] \Big)
			- h_i \frac{2 \eta}{3} \sum_{\ell=1}^{s} \E\big[\|e_{\ell}^i\|_{V_{\ell}}^2\big]\\
			&\leq h_i^{1 + 2\nu} \frac{3 L_A^2 s}{2 \eta} \Big(\frac{1 + C_{\textup{reg}}}{1 + 2 \nu}
			+ C_{\textup{reg}} \Big)
			- h_i \frac{2 \eta}{3} \sum_{\ell=1}^{s} \E\big[\|e_{\ell}^i\|_{V_{\ell}}^2\big]\\
			&=: h_i^{1 + 2\nu} C_{\Gamma_3}
			- h_i \frac{2 \eta}{3} \sum_{\ell=1}^{s} \E\big[\|e_{\ell}^i\|_{V_{\ell}}^2\big],
	\end{align*}}
	where we again used the regularity condition from \eqref{eq:convergence_proof_reg}.
	For the  term $\Gamma_4$, we add and subtract $e^{i-1}$ to the right side of the inner product. Using Lemma~\ref{lem:CauchyInEq} with $\xi = 3s^2$, we find that
	\begin{align*}
		\Gamma_4
		&= \sum_{\ell=1}^{s} \innerB[H]{\int_{t_{i-1}}^{t_i} (B_{\ell}(t,  u(t)) - B_{\ell}(t_{i-1},U^{i-1}) )\diff{W(t)}}{e_{\ell}^i - e^{i-1}} \\
		&\quad + \sum_{\ell=1}^{s} \innerB[H]{\int_{t_{i-1}}^{t_i} B_{\ell}(t,  u(t)) - B_{\ell}(t_{i-1},U^{i-1}) \diff{W(t)}}{e^{i-1}}\\
		&\leq 3 s^2 \sum_{\ell=1}^{s} \Big \| \int_{t_{i-1}}^{t_i} \big(B_{\ell}(t,  u(t)) - B_{\ell}(t_{i-1},U^{i-1}) \big) \diff{W(t)} \Big\|_H^2\\ 
		&\quad + \frac{1}{12 s^2} \sum_{\ell=1}^{s} \| e_{\ell}^i - e^{i-1}\|_H^2 \\
		&\quad + \sum_{\ell=1}^{s} \innerB[H]{\int_{t_{i-1}}^{t_i} B_{\ell}(t,  u(t)) - B_{\ell}(t_{i-1},U^{i-1}) \diff{W(t)}}{e^{i-1}}\\
		&= \Gamma_{4,1} + \Gamma_{4,2} + \Gamma_{4,3}.
	\end{align*}
	In expectation, the term $\Gamma_{4,1}$ can be bounded by first using the It\^o-isometry and then Assumption~\ref{ass:B} to obtain%
	{\allowdisplaybreaks
		\begin{align*}
			\E \big[\Gamma_{4,1} \big]
			&= 3 s^2 \sum_{\ell=1}^{s} \int_{t_{i-1}}^{t_i} \E \big[ \big \| B_{\ell}(t,  u(t)) - B_{\ell}(t_{i-1},U^{i-1}) \big\|_{L_2^0}^2 \big] \diff{t} \\
			&\leq 6 s^2 \Big(\sum_{\ell=1}^{s} \int_{t_{i-1}}^{t_i} \E \big[ \big \| B_{\ell}(t,  u(t)) - B_{\ell}(t_{i-1},u(t_{i-1})) \big\|_{L_2^0}^2 	\big] \diff{t} \\
			&\qquad + h_i \sum_{\ell=1}^{s} \E \big[ \big \| B_{\ell}(t_{i-1},u(t_{i-1}) - B_{\ell}(t_{i-1},U^{i-1}) \big\|_{L_2^0}^2 \big]\Big)\\
			&\leq 6 L_B^2 s^2 \Big(\int_{t_{i-1}}^{t_i} \big( | t - t_{i-1}|^{2\nu} + \E \big[ \| u(t) - u(t_{i-1}) \|_H^2 \big]\big) \diff{t} 
			+ h_i \E \big[ \big \| e^{i-1} \big\|_H^2 \big]\Big)\\
			&\leq 6 L_B^2 s^2 \Big(h_i^{1 + 2 \nu} \frac{1 + C_{\textup{reg}}}{1 + 2 \nu}
			+ h_i \sum_{\ell=1}^{s} \E \big[ \big \| e_{\ell-1}^{i} \big\|_H^2 \big]\Big)\\
			&=: h_i^{1 + 2 \nu} C_{\Gamma_4,1}
			+ h_i C_{\Gamma_4,2} \sum_{\ell=1}^{s} \E \big[ \big \| e_{\ell-1}^{i} \big\|_H^2 \big],
	\end{align*}}
	where we used the regularity condition from \eqref{eq:convergence_proof_reg} and the fact that $\|e^{i-1}\|_H^2 = \|e_0^{i}\|_H^2 \leq \sum_{\ell=1}^{s} \| e_{\ell-1}^{i} \|_H^2$. For the second term $\Gamma_{4 ,2}$, we apply a telescopic sum argument and H\"older's inequality for sums. Then we find that
	\begin{align*}
		\Gamma_{4 ,2}
		&= \frac{1}{12 s^2} \sum_{\ell=1}^{s} \Big\| \sum_{j=1}^{\ell} \big(e_{j}^i - e_{j-1}^{i}\big) \Big\|_H^2
		\leq \frac{1}{12 s^2} \sum_{\ell=1}^{s} \ell \sum_{j=1}^{\ell} \big\| e_{j}^i - e_{j-1}^{i}\big\|_H^2\\
		&\leq \frac{1}{12 s^2} \sum_{\ell=1}^{s} s \sum_{j=1}^{s} \big\| e_{j}^i - e_{j-1}^{i}\big\|_H^2 = \frac{1}{12} \sum_{j=1}^{s} \big\| e_{j}^i - e_{j-1}^{i}\big \|_H^2.
	\end{align*}
	The last term $\Gamma_{4,3}$ disappears in expectation. This can be seen by the stochastic independence and properties of the It\^o integral,
	\begin{align*}
		&\E\big[\Gamma_{4,3}\big] \\
		&=\E \Big[\sum_{\ell=1}^{s} \innerB[H]{\int_{t_{i-1}}^{t_i} B_{\ell}(t,  u(t)) - B_{\ell}(t_{i-1},U^{i-1}) \diff{W(t)}}{e^{i-1}} \Big]\\
		&=\E \Big[\sum_{\ell=1}^{s} \innerB[H]{\E \Big[\int_{t_{i-1}}^{t_i} B_{\ell}(t,  u(t)) - B_{\ell}(t_{i-1},U^{i-1}) \diff{W(t)}\Big| \F_{t_{i}} \Big]}{e^{i-1}}\Big| \F_{t_{i-1}} \Big] \Big]
		= 0.
	\end{align*}
	It remains to bound $\Gamma_5$. We begin to add and subtract $e^{i-1}$ in the right side of the inner product, insert the expansion for the $Q$-Wiener process and apply Lemma~\ref{lem:CauchyInEq} for $\xi = 3 s^2$. It then follows that 
	\allowdisplaybreaks{
		\begin{align*}
			\Gamma_5
			&= \sum_{\ell=1}^{s} \innerb[H]{B_{\ell}(t_{i-1},U^{i-1}) \big(\Delta W^{i} - \Delta W^{i, K_{\ell}} \big)}{e_{\ell}^i - e^{i-1}}\\
			&\quad + \sum_{\ell=1}^{s} \innerb[H]{B_{\ell}(t_{i-1},U^{i-1}) \big(\Delta W^{i} - \Delta W^{i, K_{\ell}} \big)}{e^{i-1}}\\
			&\leq 3 s^2 \sum_{\ell=1}^{s} \big\| B_{\ell}(t_{i-1},U^{i-1}) \big(\Delta W^{i} - \Delta W^{i, K_{\ell}}\big) \big\|_H^2 
			+ \frac{1}{12 s^2} \sum_{\ell=1}^{s} \| e_{\ell}^i - e^{i-1}\|_H^2 \\
			&\quad + \sum_{\ell=1}^{s} \innerb[H]{B_{\ell}(t_{i-1},U^{i-1}) \big(\Delta W^{i} - \Delta W^{i, K_{\ell}}\big) }{e^{i-1}}
			= \Gamma_{5,1} + \Gamma_{5,2} + \Gamma_{5,3}.
	\end{align*}}%
	For the first term $\Gamma_{5,1}$, we note that we can exchange $\Delta W^i - \Delta W^{i,K_{\ell}}$ by the sum $\sum_{k > K_{\ell}} \int_{t_{i-1}}^{t_{i}} \diff{\beta_{k}(t)}$ of scalar Wiener processes and recal that $\{q_{k}^{\frac{1}{2}} \psi_{k}\}_{k \in \N}$ is an orthonormal basis of $Q^{\frac{1}{2}}(H)$. Further, we recall that all these summands are stochastically independent, apply It\^o's isometry, the bound on the truncation error and the a prioiri from Lemma~\ref{lem:Apriori} to find that%
	{\allowdisplaybreaks
		\begin{align*}
			\E \big[\Gamma_{5,1}\big]
			&= 3 s^2 \sum_{\ell=1}^{s} \E\Big[ \Big\| \sum_{k > K_{\ell}} \int_{t_{i-1}}^{t_{i}} B_{\ell}(t_{i-1}, U^{i-1}) \big(q_{k}^{\frac{1}{2}} \psi_{k} \big) \diff{\beta_{k}(t)} \Big\|_H^2 \Big]\\
			&= 3 s^2 \sum_{\ell=1}^{s} \sum_{k > K_{\ell}} \E\Big[ \Big\| \int_{t_{i-1}}^{t_{i}} B_{\ell}(t_{i-1}, U^{i-1}) 	\big(q_{k}^{\frac{1}{2}} \psi_{k} \big) \diff{\beta_{k}(t)} \Big\|_H^2 \Big]\\
			&= h_i 3 s^2 \sum_{\ell=1}^{s} \sum_{k > K_{\ell}} \E\big[ \big\| B_{\ell}(t_{i-1}, U^{i-1}) \big(q_{k}^{\frac{1}{2}} \psi_{k}\big) \big\|_H^2 \big]\\
			&\leq h_i 3 s^2 \Big(\sum_{\ell=1}^{s} \varepsilon_{\ell}\Big) \big(1 + \E\big[ \|U^{i-1} \|_H^2 \big]\big)
			\leq h_i 3 s^2 (1 + C_{\textup{bound}}) \sum_{\ell=1}^{s} \varepsilon_{\ell}
			=: h_i \varepsilon C_{\Gamma_5}.
	\end{align*}}%
	Analogously to $\Gamma_{4,2}$ and $\Gamma_{4,3}$, it follows that $\Gamma_{5,2} \leq \frac{1}{12} \sum_{\ell=1}^{s} \big\| e_{\ell}^i - e_{\ell-1}^{i}\big \|_H^2$ and $\E \big[ \Gamma_{5,3} \big] = 0$. We can now combine the previous bounds to the following,%
	{\allowdisplaybreaks
		\begin{align*}
			&\frac{1}{2} \Big( \E\big[\| e^i\|_H^2\big] - \E\big[\| e^{i-1}\|_H^2\big] + \sum_{\ell=1}^{s} \E\big[\| e_{\ell}^i - 	e_{\ell-1}^i\|_H^2\big]\Big)\\
			&\leq h_i^{1+ 2\nu} \big(C_{\Gamma_1,1} + C_{\Gamma_2} + C_{\Gamma_3} + C_{\Gamma_4,1}\big)
			+ h_i \big(C_{\Gamma_1,2}+C_{\Gamma_4,2}\big) \sum_{\ell=1}^{s} \E \big[ \| e_{\ell-1}^i\|_H^2\big]\\
			&\quad- h_i \frac{\eta}{2} \sum_{\ell=1}^{s} \E\big[\|e_{\ell}^i\|_{V_{\ell}}^2\big]
			+ h_i \varepsilon C_{\Gamma_5}
			+ \frac{1}{4} \sum_{j=1}^{s} \E \big[\big\| e_{j}^i - e_{j-1}^{i}\big \|_H^2\big].
	\end{align*}}%
	After multiplying by $2$, rearranging some terms and summing up from $i = 1$ to $n$, we find
	\begin{align*}
		&\E\big[\| e^n\|_H^2\big] - \E\big[\| e^{0}\|_H^2\big] + \frac{1}{2} \sum_{i=1}^{n} \sum_{\ell=1}^{s} \E\big[\| e_{\ell}^i - e_{\ell-1}^i\|_H^2\big]
		+ \eta \sum_{i=1}^{n} h_i \sum_{\ell=1}^{s} \E\big[\|e_{\ell}^i\|_{V_{\ell}}^2\big]\\
		&\leq h^{2\nu} 2T \big(C_{\Gamma_1,1} + C_{\Gamma_2} + C_{\Gamma_3} + C_{\Gamma_4,1}\big)
		+ 2 \big(C_{\Gamma_1,2}+C_{\Gamma_4,2}\big) \sum_{i=1}^{n} h_i \sum_{\ell=1}^{s} \E \big[ \| e_{\ell-1}^i\|_H^2\big]\\
		&+ \varepsilon 2 T C_{\Gamma_5}.
	\end{align*}
	In the last step, we use that the initial error $\E\big[\| e^{0}\|_H^2\big]$ is zero and it then remains to apply Gr\"onwall's inequality for sums. We abbreviate the terms as follows
	\begin{align*}
		u_{n,\ell} &= \E\big[\|e_{\ell}^n\|_H^2\big] + \frac{1}{2} \sum_{i = 1}^{n} \sum_{j=1}^{\ell} \E\big[\| e_{j}^i - e_{j-1}^i\|_H^2\big]
		+ \eta \sum_{i = 1}^{n} h_i \sum_{j=1}^{\ell} \E\big[\|e_{j}^i\|_{V_{j}}^2\big],\\
		a &= h^{2\nu} 2T \big( C_{\Gamma_1,1} + C_{\Gamma_2} + C_{\Gamma_3} + C_{\Gamma_4,1}\big) + \varepsilon 2 T C_{\Gamma_5},
		\quad b_{i,\ell} = h_i 2 \big( C_{\Gamma_1,2} + C_{\Gamma_4,2} \big).
	\end{align*}
	Recall that $\E\big[\|e_{s}^n\|_H^2\big] = \E\big[\|e^n\|_H^2\big]$. 
	Applying Lemma~\ref{lem:discreteGronwall2} then shows that $u_{n, s} \leq a \exp \big( \sum_{i=1}^{n-1} \sum_{\ell=1}^{s} b_{i, \ell} \big)$ or equivalently
	\begin{align*}
		&\E\big[\|e^n\|_H^2\big] + \frac{1}{2} \sum_{i = 1}^{n} \sum_{\ell=1}^{s} \E\big[\| e_{\ell}^i - e_{\ell-1}^i\|_H^2\big]
		+ \eta \sum_{i = 1}^{n} h_i \sum_{\ell=1}^{s} \E\big[\|e_{\ell}^i\|_{V_{\ell}}^2\big]\\
		&\leq \exp \big(2 s T \big( C_{\Gamma_1,2} + C_{\Gamma_4,2} \big)\big)
		\big(h^{2\nu} 2T \big( C_{\Gamma_1,1} + C_{\Gamma_2} + C_{\Gamma_3} + C_{\Gamma_4,1}\big) + \varepsilon 2 T C_{\Gamma_5}\big),
	\end{align*}
	which is our desired bound.
\end{proof}

\section{Numerical experiment}
\label{sec:numerical_experiment}

In our numerical experiment, we consider a stochastic heat equation on a domain $\D = (0,1)^2$ in the two dimensional space $\R^2$. We allow for a nonlinear perturbation and  possibly multiplicative noise. More exactly, we look at an SPDE of the form
\begin{align}\label{eq:SPDE_numEx}
	\begin{cases}
		\diff{u(t,x)} - \alpha(t) \Delta u(t,x) \diff{t} = \big( \tilde{F}(t, x, u) + \tilde{G}(t, x)\big) \diff{t}\\  
		\hspace{4.5cm} + \tilde{B}(t, x, u) \diff{W(t,x)}, \quad
		t \in (0,1], x \in \D,\\
		u(t, x) = 0, \hspace{6.35cm} t \in [0,1], x \in \partial \D,\\
		u(0,x) = u_0(x), \hspace{5.6cm} x \in \D.
	\end{cases}
\end{align}
In the following, we denote the Hilbert spaces $V = W_0^{1,2}(\D)$, $H = L^2(\D)$ and $V^* = \big(W_0^{1,2}(\D)\big)^{*}$. The coefficients $\alpha$, $\tilde{F}$, $\tilde{G}$ and $\tilde{B}$, will be explained in more detail in Sections~\ref{subsec:experiment1} and \ref{subsec:experiment2}. We then consider the abstract functions and operators as defined in \eqref{eq:defA}--\eqref{eq:defB}. These operators fulfill the Assumptions~\ref{ass:initialValue}, \ref{ass:opA}, \ref{ass:F} and \ref{ass:B} as explained in Sections~\ref{subsec:domain_decomp} and \cite[Lemma~3.1]{EisenmannStillfjord.2022}.

The noise $W$ is a $Q$-Wiener process where the eigenvalues $\{q_k\}_{k \in \N^2}$ and eigenfunctions $\{\psi_k\}_{k \in \N^2}$ of $Q$ are given by
\begin{align*}
	q_{k} = (k_1^2 +k_2^2)^{- \frac{1}{2} - \delta} 
	\quad \text{and} \quad
	\psi_k(x) = 2 \sin(k_1 \pi x_1) \sin(k_2 \pi x_2)
\end{align*}
for $\delta \in (0,\infty)$. 
Then we can represent $W$ and the truncated noise $W^K$ through the Karhunen-Lo\`{e}ve expansion (cf. \cite[Section~10.2]{LordEtAl.2014})
\begin{align*}
	W(t, x) = \sum_{k \in \N^2} q_k^{\frac{1}{2}} \beta_{k}(t) \psi_{k}(x) \quad \text{and} \quad
	W^K(t, x) = \sum_{|k| \leq K} q_k^{\frac{1}{2}} \beta_{k}(t) \psi_{k}(x),
\end{align*}
where $\{\beta_k(t)\}_{k \in \N^2}$ are independent and identically distributed $\F_t$-Wiener processes, $|k| = |(k_1, k_2)| = \max\{|k_1|, |k_2|\}$ and $K \in \N$. The functions $\{\psi_{k}\}_{k \in \N^2}$ form  an orthonormal basis of $H = L^2(\D)$ and they are eigenfunctions of the Laplace operator with homogeneous Dirichlet boundary conditions. The noise therefore fits into the setting explained in Assumption~\ref{ass:W}.

For $\alpha \equiv 1$ and sufficiently smooth $\tilde{F}$, $\tilde{B}$ and $u_0$, we can verify that the exact solution lies in $C^{\nu}([0,T]; L^2(\Omega;V))$ for $\nu \in (0,\delta)$ using Theorem~\ref{thm:regularity}. When choosing a very small value $\delta$ close to zeros, we should therefore only observe a convergence order close to zero. This is a certain gap between our error bound and the numerical results in the coming subsections. While similar regularity assumptions have been made in \cite{KruseWeiske.2021}, they seem to observe the same gap as we do. A possible explanation could be that the example still provides a more regular solution or that the same bound could also be observed without the H\"older condition with values in $V$ for the exact solution. This is a question left for future research.

For the domain decomposition, we divide the domain $(0,1)^2$ into four strips in $x_1$-direction. We choose the weight functions $\chi_{\ell}$, $\ell \in \{1,\dots,s\}$, as piecewise linear functions with an overlap that is of the size $1/10$, compare Figure~\ref{fig:weight_functions} for a visualisation. 
\begin{figure}
	\includegraphics[width=0.45\textwidth]{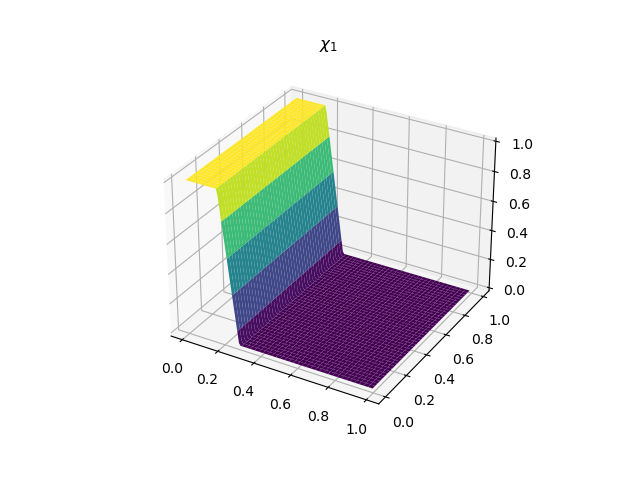}
	\includegraphics[width=0.45\textwidth]{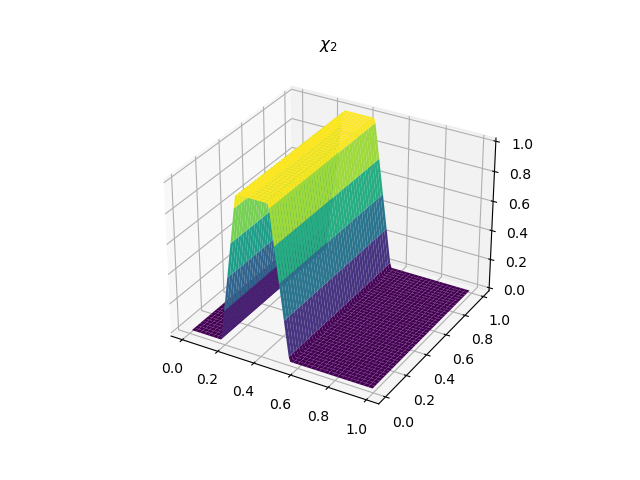}
	\includegraphics[width=0.45\textwidth]{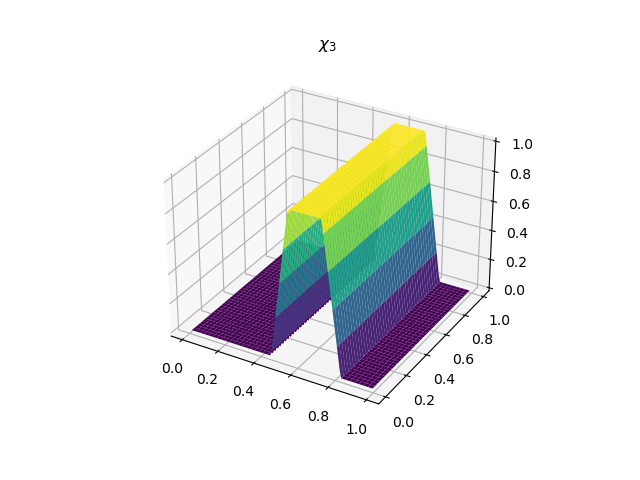}
	\includegraphics[width=0.45\textwidth]{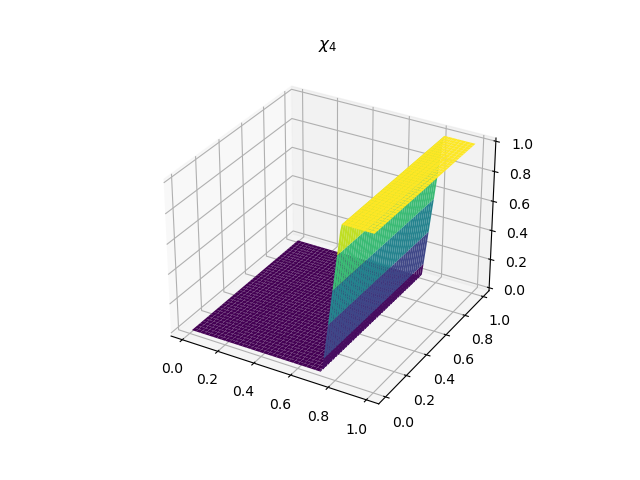}
	\caption{Visuablization of the weight functions $\chi_1, \chi_2, \chi_3$ and $\chi_4$}
	\label{fig:weight_functions}
\end{figure}
Using these weight functions, we can introduce the split data as previously defined in \eqref{eq:defA}--\eqref{eq:defB}.

In our numerical experiment, we discretize \eqref{eq:SPDE_numEx} using the time discrtization scheme \eqref{eq:scheme}. Moreover, to discretize the equation in space, we use a standard finite difference approximation, where we subdivide the domain in both $x_1$ and $x_2$ direction in $128$ points. The noise is truncated with $K = 128$ and we choose regularisation parameter $\delta = 0.001$. In this setting, the spatial error and truncation error are small enough such that the temporal error will be dominating the error.
In order to model the expectation involved in the mean-square error bound, we use $50$ Monte Carlo iterations.

\subsection{Experiment 1}
\label{subsec:experiment1}

In the first example, we choose the functions
\begin{align*}
	u_0(x) &= 5 x_1^2 (x_1-1)^3 x_2^2 (x_2-1)^3, \quad \alpha(t) = 0.1 (1 + \exp(-t))\\
	\tilde{G}(t,x) &= 0, \quad \tilde{F}(t, x, u) = \sin(u), \quad \tilde{B}(t, x, u) = 1.
\end{align*}
The function $u_0$ is a smooth and bounded function on $\D$ and therefore in particular in $H$. As it does not depend on $\omega$, it follows directly that $u_0 \in L^2(\Omega;H)$. Moreover, the function $\alpha$ is a smooth function and the functions $\tilde{F}$, $\tilde{G}$ and $\tilde{B}$ do not depend on $t$ in the following. 

Thus, all functions satisfy the given assumptions on the time-regularity stated in Theorem~\ref{thm:errorBound}. The abstract function $G$ is zero as can be seen as $\tilde{G}$ is zero which is inserted into \eqref{eq:defG}. Moreover, since $\sin$ is a Lipschitz continuous function, it follows that the abstract function $F$ given through $\tilde{F}$ in \eqref{eq:defF} is Lipschitz continuous in the second argument. Similarly, the abstract function $B$ given through $\tilde{B}$ in \eqref{eq:defB} is Lipschitz continuous in the second argument as it is constant. Moreover, since the sequence $\{(k_1^2 +k_2^2)^{-\frac{1}{2}-\delta}\}_{k \in \N^2}$ is summable, it follows that for every $\varepsilon_{\ell} > 0$, we can find a $K_{\ell} \in \N$ such that
\begin{align*}
	\sum_{|k| > K_{\ell}} \E\big[ \|B_{\ell}(t,v)(q_{k}^{\frac{1}{2}} \psi_k)\|_H^2 \big]
	\leq \sum_{|k| > K_{\ell}} q_k = \sum_{|k| > K_{\ell}} (k_1^2 +k_2^2)^{-\frac{1}{2}-\delta} \leq \varepsilon_{\ell}.
\end{align*}
for every $v \in L^2(\Omega;H)$. 

Altogether, we have a semi-linear equation with additive noise that fits into the setting of Theorem~\ref{thm:errorBound} and observe a convergence rate of $\frac{1}{2}$ in the numerical test presented on the left side of Figure~\ref{fig:num_exp}.

\begin{figure}
	\includegraphics[width=0.49\textwidth]{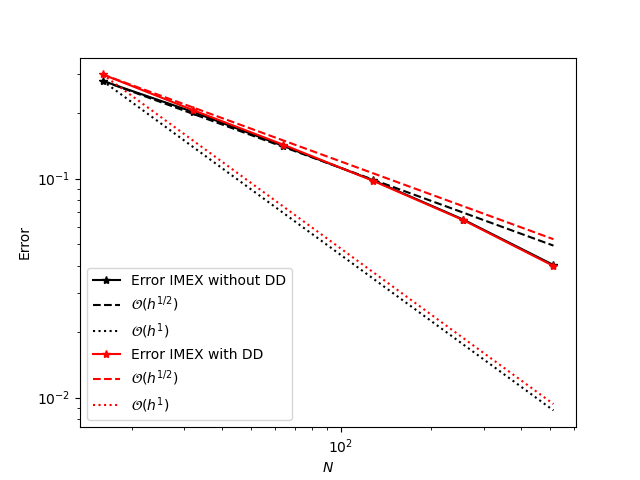}
	\includegraphics[width=0.49\textwidth]{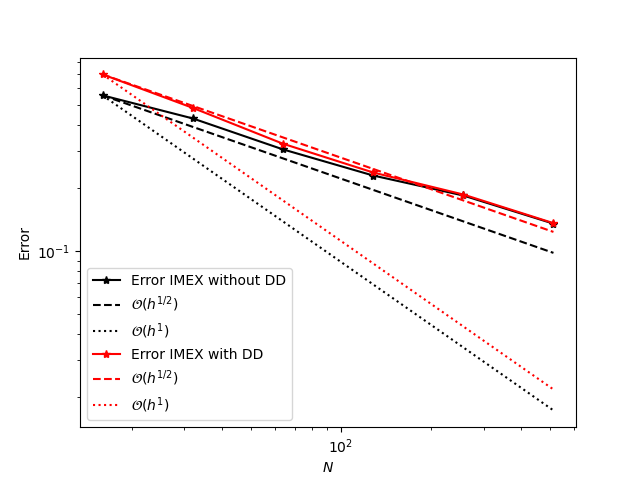}
	\caption{Comparison of an IMEX-Euler--Maruyama scheme with and without a domain decomposition. The left-hand side shows the result of Experiment~1 and right-hand side the results of Experiment~2.}
	\label{fig:num_exp}
\end{figure}

\subsection{Experiment 2}
\label{subsec:experiment2}

In the second example, we consider a linear equation with multiplicative noise with the coefficients
\begin{align*}
	u_0(x) &= 5 x_1^2 (x_1-1)^3 x_2^2 (x_2-1)^3, \quad \alpha(t) = 0.1, \quad \tilde{G}(t,x) = 0,\\
	\tilde{F}(t, x, u) &= (3 \sin(5x_1) + 2 \cos(7x_2)) (\cos(t) + \sin(4t)), \quad 
	\tilde{B}(t, x, u) = u.
\end{align*}
We can argue in a similar fashion to that in Section~\ref{subsec:experiment2} that the coefficient fit into the setting required in Theorem~\ref{thm:errorBound}. The coefficients are all smooth. In contrast to Section~\ref{subsec:experiment2} the corresponding diffusion operator $B$ is not constant in the second argument but still linear and bounded. In order to fulfill the truncation assumption (H4), we first note that $\|\psi_k\|_{\infty} \leq 2$ for every $k \in \N^2$. Since $\{(k_1 +k_2)^{-\frac{1}{2}-\delta}\}_{k \in \N^2}$ is summable, we can then again argue that for every $\varepsilon_{\ell} > 0$, there exists $K_{\ell} \in \N$ such that
\begin{align*}
	\sum_{|k| > K_{\ell}} \E\big[ \|B_{\ell}(t,v)(q_{k}^{\frac{1}{2}} \psi_k)\|_H^2 \big]
	\leq 4 \sum_{|k| > K_{\ell}} (k_1^2 +k_2^2)^{-\frac{1}{2}-\delta} \E\big[\| v \|_H^2 \big]
	\leq \varepsilon_{\ell} \E\big[\| v \|_H^2 \big].
\end{align*}
for every $v \in L^2(\Omega;H)$. Thus, the truncation assumption from (H4) is fulfilled. 
Again, we observe a convergence order of $\frac{1}{2}$ in the right error plot from Figure~\ref{fig:num_exp}.

\bibliographystyle{plain}
\bibliography{main_arxiv}

\begin{thebibliography}{10}

\bibitem{Ableidinger2016}
M.~Ableidinger and E.~Buckwar.
\newblock {Splitting integrators for the stochastic Landau--Lifshitz equation}.
\newblock {\em SIAM J. Sci. Comput.}, 38:A1788--A1806, 01 2016.

\bibitem{AdamsFournier.2003}
R.~A. Adams and J.~J.~F. Fournier.
\newblock {\em Sobolev {S}paces}.
\newblock Elsevier/Academic Press, Amsterdam, second edition, 2003.

\bibitem{AlamoSanzSerna.2016}
A.~Alamo and J.~M. Sanz-Serna.
\newblock A technique for studying strong and weak local errors of splitting
  stochastic integrators.
\newblock {\em SIAM J. Numer. Anal.}, 54(6):3239--3257, 2016.

\bibitem{alouges2014}
F.~Alouges, A.~de~Bouard, and A.~Hocquet.
\newblock A semi-discrete scheme for the stochastic {L}andau-{L}ifshitz
  equation.
\newblock {\em Stoch. Partial Differ. Equ. Anal. Comput.}, 2(3):281--315, 2014.

\bibitem{ArrarasEtAl.2017}
A.~Arrar\'{a}s, K.~J. in~'t Hout, W.~Hundsdorfer, and L.~Portero.
\newblock Modified {D}ouglas splitting methods for reaction-diffusion
  equations.
\newblock {\em BIT}, 57(2):261--285, 2017.

\bibitem{banas2014}
L.~Ba\v{n}as, Z.~Brze\'{z}niak, M.~Neklyudov, and A.~Prohl.
\newblock {\em Stochastic ferromagnetism}, volume~58 of {\em De Gruyter Studies
  in Mathematics}.
\newblock De Gruyter, Berlin, 2014.
\newblock Analysis and numerics.

\bibitem{Brehier.2023}
C.-E. Br\'{e}hier and D.~Cohen.
\newblock Analysis of a splitting scheme for a class of nonlinear stochastic
  {S}chr\"{o}dinger equations.
\newblock {\em Appl. Numer. Math.}, 186:57--83, 2023.

\bibitem{Prohl.2012}
E.~Carelli, A.~M\"{u}ller, and A.~Prohl.
\newblock Domain decomposition strategies for the stochastic heat equation.
\newblock {\em Int. J. Comput. Math.}, 89(18):2517--2542, 2012.

\bibitem{Clark.1987}
D.~S. Clark.
\newblock Short proof of a discrete {G}ronwall inequality.
\newblock {\em Discrete Appl. Math.}, 16(3):279--281, 1987.

\bibitem{cong2014}
C.~Cong, X.-C. Cai, and K.~Gustafson.
\newblock Implicit space-time domain decomposition methods for stochastic
  parabolic partial differential equations.
\newblock {\em SIAM J. Sci. Comput.}, 36(1):C1--C24, 2014.

\bibitem{lord2013}
E.~J. Coutts and G.~J. Lord.
\newblock Effects of noise on models of spiny dendrites.
\newblock {\em J. Comput. Neurosci.}, 34(2):245--257, 2013.

\bibitem{EisenmannPhD.2019}
M.~Eisenmann.
\newblock {\em Methods for the Temporal Approximation of Nonlinear,
  Nonautonomous Evolution Equations}.
\newblock {PhD} thesis, TU Berlin, 2019.

\bibitem{EisenmannHansen.2018}
M.~Eisenmann and E.~Hansen.
\newblock Convergence analysis of domain decomposition based time integrators
  for degenerate parabolic equations.
\newblock {\em Numer. Math.}, 140(4):913--938, 2018.

\bibitem{EisenmannHansen.2022}
M.~Eisenmann and E.~Hansen.
\newblock A variational approach to the sum splitting scheme.
\newblock {\em IMA J. Numer. Anal.}, 42(1):923--950, 2022.

\bibitem{EisenmannStillfjord.2022}
M.~Eisenmann and T.~Stillfjord.
\newblock A randomized operator splitting scheme inspired by stochastic
  optimization methods.
\newblock {\em ArXiv Preprint, arXiv:2210.05375}, 2022.

\bibitem{Evans.1998}
L.~C. Evans.
\newblock {\em Partial {D}ifferential {E}quations}.
\newblock American Mathematical Society, Providence, RI, 1998.

\bibitem{geiger2012}
S.~Geiger, G.~Lord, and A.~Tambue.
\newblock Exponential time integrators for stochastic partial differential
  equations in 3d reservoir simulation.
\newblock {\em Comput. Geosci.}, 16(2):323--334, 2012.

\bibitem{Geiser.2009}
J.~Geiser and C.~Kravvaritis.
\newblock A domain decomposition method based on the iterative operator
  splitting method.
\newblock {\em Appl. Numer. Math.}, 59(3-4):608--623, 2009.

\bibitem{GlowinskiEtAl.2016}
R.~Glowinski, S.~J. Osher, and W.~Yin, editors.
\newblock {\em Splitting methods in communication and imaging, science, and
  engineering}.
\newblock Sci. Comput. Cham: Springer, 2016.

\bibitem{gyoengy1998}
I.~Gy\"{o}ngy.
\newblock Lattice approximations for stochastic quasi-linear parabolic partial
  differential equations driven by space-time white noise. {I}.
\newblock {\em Potential Anal.}, 9(1):1--25, 1998.

\bibitem{HansenHenningsson.2017}
E.~Hansen and E.~Henningsson.
\newblock Additive domain decomposition operator splittings---convergence
  analyses in a dissipative framework.
\newblock {\em IMA J. Numer. Anal.}, 37(3):1496--1519, 2017.

\bibitem{hausenblas2002}
E.~Hausenblas.
\newblock Numerical analysis of semilinear stochastic evolution equations in
  {B}anach spaces.
\newblock {\em J. Comput. Appl. Math.}, 147(2):485--516, 2002.

\bibitem{Ji.2023}
L.~Ji.
\newblock An overlapping domain decomposition splitting algorithm for
  stochastic nonlinear {S}chr\"{o} dinger equation.
\newblock {\em ArXiv Preprint, arXiv:2309.03393}, 2023.

\bibitem{jin2007}
C.~Jin, X.-C. Cai, and C.~Li.
\newblock Parallel domain decomposition methods for stochastic elliptic
  equations.
\newblock {\em SIAM J. Sci. Comput.}, 29(5):2096--2114, 2007.

\bibitem{Kruse.2014}
R.~Kruse.
\newblock {\em Strong and weak approximation of semilinear stochastic evolution
  equations}, volume 2093 of {\em Lecture Notes in Mathematics}.
\newblock Springer, Cham, 2014.

\bibitem{KruseLarsson.2012}
R.~Kruse and S.~Larsson.
\newblock Optimal regularity for semilinear stochastic partial differential
  equations with multiplicative noise.
\newblock {\em Electron. J. Probab.}, 17:no. 65, 19, 2012.

\bibitem{KruseWeiske.2021}
R.~Kruse and R.~Weiske.
\newblock The {BDF}2-{M}aruyama scheme for stochastic evolution equations with
  monotone drift.
\newblock {\em ArXiv Preprint, arXiv:2105.08767}, 2021.

\bibitem{lindstroem2011}
F.~Lindgren, H.~Rue, and J.~Lindstr\"{o}m.
\newblock An explicit link between {G}aussian fields and {G}aussian {M}arkov
  random fields: the stochastic partial differential equation approach.
\newblock {\em J. R. Stat. Soc. Ser. B Stat. Methodol.}, 73(4):423--498, 2011.
\newblock With discussion and a reply by the authors.

\bibitem{LiuRoeckner.2015}
W.~Liu and M.~R\"{o}ckner.
\newblock {\em Stochastic partial differential equations: an introduction}.
\newblock Universitext. Springer, Cham, 2015.

\bibitem{LordEtAl.2014}
G.~J. Lord, C.~E. Powell, and T.~Shardlow.
\newblock {\em An introduction to computational stochastic {PDE}s}.
\newblock Cambridge Texts in Applied Mathematics. Cambridge University Press,
  New York, 2014.

\bibitem{MathewBook2008}
T.~P. Mathew.
\newblock {\em Domain decomposition methods for the numerical solution of
  partial differential equations}, volume~61 of {\em Lect. Notes Comput. Sci.
  Eng.}
\newblock Berlin: Springer, 2008.

\bibitem{MclachlanQuispel.2002}
R.~Mclachlan and G.~Quispel.
\newblock {Splitting methods}.
\newblock {\em Acta Numer.}, 11:341--434, 01 2002.

\bibitem{Roubicek.2013}
T.~Roub\'{\i}\v{c}ek.
\newblock {\em Nonlinear partial differential equations with applications},
  volume 153 of {\em International Series of Numerical Mathematics}.
\newblock Birkh\"{a}user/Springer Basel AG, Basel, second edition, 2013.

\bibitem{stannat2016}
M.~Sauer and W.~Stannat.
\newblock Analysis and approximation of stochastic nerve axon equations.
\newblock {\em Math. Comput.}, 85(301):2457--2481, 2016.

\bibitem{shardlow1999}
T.~Shardlow.
\newblock Numerical methods for stochastic parabolic {PDE}s.
\newblock {\em Numer. Funct. Anal. Optim.}, 20(1-2):121--145, 1999.

\bibitem{Shardlow.2003}
T.~Shardlow.
\newblock Splitting for dissipative particle dynamics.
\newblock {\em SIAM J. Sci. Comput.}, 24(4):1267--1282, 2003.

\bibitem{Vabishchevich.1989}
P.~N. Vabishchevich.
\newblock Difference schemes with domain decomposition for solving
  non-stationary problems.
\newblock {\em U.S.S.R. Comput. Math. Math. Phys.}, 29(6):155--160, 1989.

\bibitem{veraar2006stochastic}
M.~C. Veraar.
\newblock Stochastic integration in banach spaces and applications to parabolic
  evolution equations.
\newblock {\em PhD Thesis, Delft University}, 2006.

\bibitem{zhang2008}
K.~Zhang, R.~Zhang, Y.~Yin, and S.~Yu.
\newblock Domain decomposition methods for linear and semilinear elliptic
  stochastic partial differential equations.
\newblock {\em Applied Mathematics and Computation}, 195(2):630--640, 2008.

\end{thebibliography}

\appendix
\section{Basic inequalities} \label{sec:appendix}

\begin{lemma} \label{lem:CauchyInEq}
	For all $x,y \in \R$, $\xi > 0$ it follows that $|xy| \leq \xi|x|^2 + \frac{1}{4 \xi}|y|^2$.
\end{lemma}

A proof can be found \cite[B.2. Elementary inequalities]{Evans.1998}.
Further, we need a form of Gr\"onwall's inequality in this paper. A proof for the standard statement can be found in \cite{Clark.1987}. We use a variant that is based on a rearrangement of the terms, where the summands are subdivided into two sums.

\begin{lemma} \label{lem:discreteGronwall2}
	Let $(u_{n, \ell})_{n\in \N, \ell \in \{1,\dots, s \}}$ and $(b_{n, \ell})_{n \in \N, \ell \in \{1,\dots, s \}}$ be two nonnegative sequences
	that satisfy, for given $a\in [0,\infty)$ and $N \in \N$, that
	\begin{align*}
		u_{n, s} \leq a + \sum_{i=1}^{n-1} \sum_{\ell=1}^{s} b_{i, \ell} u_{i, \ell}, 
		\quad n \in \{1,\dots, N \}.
	\end{align*}	
	Then it follows that
	\begin{align*}
		u_{n, s} \leq a \exp \Big( \sum_{i=1}^{n-1} \sum_{\ell=1}^{s} b_{i, \ell} \Big), 
		\quad n \in \{1,\dots, N \}.
	\end{align*}
\end{lemma}

\section{Proof of moment bound}\label{sec:appendix_proof}

For the sake of completeness, we state the proof of Lemma~\ref{lem:Apriori} in the following.

\begin{proof}[Proof of Lemma~\ref{lem:Apriori}]
	The idea of this proof follows the idea of an a priori bound that is commonly used for the Rothe methode for deterministic (nonlinear) evolution equations. A common step for variational solutions, we start by testing the equation given by the scheme with the solution and use the coercivity and Lipschitz continuity of the operators. A similar proof structure can be found in \cite[Lemma~8.6]{Roubicek.2013}. First, we consider a sub-step of the numerical scheme
	\begin{align*}
		U_{\ell}^i - U_{\ell-1}^i = h_i G_{\ell}(t_i) 
		+ h_i F_{\ell}(t_{i-1}, U_{\ell-1}^i)
		- h_i A_{\ell}(t_i, U_{\ell}^i)
		+ B_{\ell}(t_{i-1}, U^{i-1}) \Delta W^{i, K_{\ell}},
	\end{align*}
	which we test with the solution $U_{\ell}^i$. Then we use the identity $\inner[H]{a - b}{a} = \frac{1}{2} \big( \|a\|_H^2 - \|b\|_H^2 + \|a - b\|_H^2\big)$ for $a, b \in H$ and sum up the equation from $\ell = 1$ to $s$ to obtain
	\begin{align*}
		&\frac{1}{2} \sum_{\ell=1}^{s} \big( \|U_{\ell}^i\|_H^2 - \| U_{\ell-1}^i\|_H^2 + \|U_{\ell}^i - U_{\ell-1}^i\|_H^2 \big)
		= \sum_{\ell=1}^{s} \inner[H]{U_{\ell}^i - U_{\ell-1}^i}{U_{\ell}^i}\\ 
		&= h_i \sum_{\ell=1}^{s}\inner[H]{F_{\ell}(t_{i-1}, U_{\ell-1}^i)}{U_{\ell}^i}
		+ h_i \sum_{\ell=1}^{s} \dualVell{G_{\ell}(t_i)}{U_{\ell}^i} \\
		&\quad - h_i \sum_{\ell=1}^{s} \dualVell{A_{\ell}(t_i, U_{\ell}^i)}{U_{\ell}^i}
		+ \sum_{\ell=1}^{s} \inner[H]{B_{\ell}(t_{i-1}, U^{i-1}) \Delta W^{i, K_{\ell}}}{U_{\ell}^i}\\
		&= \Gamma_1 + \Gamma_2 + \Gamma_3 + \Gamma_4.
	\end{align*}
	In the following, we bound the single terms separately. By adding and subtracting the term $U_{\ell-1}^i$ in the right side of the inner product and using the Cauchy-Schwarz inequality, the first term $\Gamma_1$ can be bounded by
	\begin{align*}
		\Gamma_1 
		&\leq h_i \sum_{\ell=1}^{s} \|F_{\ell}(t_{i-1}, U_{\ell-1}^i) \|_H \| U_{\ell}^i - U_{\ell-1}^{i}\|_H
		+ h_i \sum_{\ell=1}^{s} \| F_{\ell}(t_{i-1}, U_{\ell-1}^i) \|_H \| U_{\ell-1}^i \|_H.
	\end{align*}
	Further applying Lemma~\ref{lem:CauchyInEq} with $\xi = 2h$ and $\xi = \frac{1}{4}$ respectively, as well as the Lipschitz continuity of $F_{\ell}$ and boundedness of $F_{\ell}(t_{i-1},0)$ in the $H$-norm, we can bound $\Gamma_1$ by%
	{\allowdisplaybreaks
		\begin{align*}
			\Gamma_1
			&\leq h_i \sum_{\ell=1}^{s} \Big(2h \|F_{\ell}(t_{i-1}, U_{\ell-1}^i) \|_H^2
			+ \frac{1}{8h} \| U_{\ell}^i - U_{\ell-1}^{i}\|_H^2\Big)\\
			&\quad + h_i \sum_{\ell=1}^{s} \Big( \frac{1}{4}\| F_{\ell}(t_{i-1}, U_{\ell-1}^i) \|_H^2 + \| U_{\ell-1}^i \|_H^2 \Big)\\
			&\leq h_i \Big(4h + \frac{1}{2}\Big) \sum_{\ell=1}^{s} \|F_{\ell}(t_{i-1}, U_{\ell-1}^i) - F_{\ell}(t_{i-1}, 0)\|_H^2 
			+ h_i \Big(4h+ \frac{1}{2}\Big) s K_F^2 \\
			&\quad + \frac{1}{8} \sum_{\ell=1}^{s} \| U_{\ell}^i - U_{\ell-1}^{i}\|_H^2
			+ h_i \sum_{\ell=1}^{s} \| U_{\ell-1}^i \|_H^2\\
			&\leq h_i \Big(4h L_F^2 + \frac{L_F^2}{2} + 1\Big) \sum_{\ell=1}^{s} \|U_{\ell-1}^i \|_H^2 
			+ h_i \Big(4h+ \frac{1}{2}\Big) s K_F^2
			+ \frac{1}{8} \sum_{\ell=1}^{s} \| U_{\ell}^i - U_{\ell-1}^{i}\|_H^2,
	\end{align*}}
	where we abbreviate $C_{\Gamma_1,1} := 4h L_F^2 + \frac{L_F^2}{2} + 1$ and $C_{\Gamma_1,2} := (4h + \frac{1}{2}) s K_F^2$.
	Furthermore, for the second summand $\Gamma_2$, we start by applying Cauchy--Schwarz inequality for the duality pairing. We can then apply Lemma~\ref{lem:CauchyInEq} to decompose the product into two summands, one containing the term $\| G_{\ell}(t_i)\|_{V_{\ell}^*}$ and one that we will compensate through the coercivity of $A_{\ell}(t_i, \cdot)$. More precesely, we apply Lemma~\ref{lem:CauchyInEq} with the constant $\xi = \frac{1}{2 \mu}$ where $\mu$ is connected to the coercivity of $A_{\ell}(t_i, \cdot)$ and then obtain
	\begin{align*}
		\Gamma_2
		\leq h_i \sum_{\ell=1}^{s} \| G_{\ell}(t_i)\|_{V_{\ell}^*} \|U_{\ell}^i\|_{V_{\ell}}
		\leq h_i \frac{1}{2 \mu} \sum_{\ell=1}^{s} \| G_{\ell}(t_i)\|_{V_{\ell}^*}^2 + h_i \frac{\mu}{2} \sum_{\ell=1}^{s} \|U_{\ell}^i\|_{V_{\ell}}^2,
	\end{align*}
	where we abbreviate $C_{\Gamma_2} := \frac{1}{2 \mu} \sum_{\ell=1}^{s} \| G_{\ell}(t_i)\|_{V_{\ell}^*}^2$ in the following.
	To bound the summand $\Gamma_3$, we now apply the coercivity of $A_{\ell}(t_i, \cdot)$ and find
	\begin{align*}
		\Gamma_3 
		= - h_i \sum_{\ell=1}^{s} \dualVell{A_{\ell}(t_i, U_{\ell}^i)}{U_{\ell}^i}
		\leq - h_i \mu \sum_{\ell=1}^{s} \| U_{\ell}^i\|_{V_{\ell}}^2.
	\end{align*}
	It remains to bound the term $\Gamma_4$, which includes the Wiener process. We then add and subtract the term $\sum_{\ell=1}^{s} \inner[H]{B_{\ell}(t_{i-1}, U^{i-1}) \Delta W^{i, K_{\ell}}}{U^{i-1}}$ and again apply Lemma~\ref{lem:CauchyInEq} with $\xi = 2s^2$ to obtain
	\begin{align*}
		\Gamma_4 
		&\leq 2 s^2 \sum_{\ell=1}^{s} \| B_{\ell}(t_{i-1}, U^{i-1}) \Delta W^{i, K_{\ell}}\|_H^2
		+ \frac{1}{8s^2} \sum_{\ell=1}^{s} \| U_{\ell}^i - U^{i-1}\|_H^2\\
		&\quad + \sum_{\ell=1}^{s} \inner[H]{B_{\ell}(t_{i-1}, U^{i-1}) \Delta W^{i, K_{\ell}}}{U^{i-1}}
		= \Gamma_{4,1} + \Gamma_{4,2} + \Gamma_{4,3}.
	\end{align*}
	Next, we begin to consider the expectation of $\Gamma_{4,1}$ more closely. To this extend, we first start to bound the expectation of the summands $\| B_{\ell}(t_{i-1}, U^{i-1}) \Delta W^{i, K_{\ell}}\|_H^2$ separately. Inserting the definition of the $Q$-Wiener process, applying the independence of the Brownian motions $\beta_k$ and an application of the It\^o isometry, yields that
	\begin{align*}
		&\E\big[ \| B_{\ell}(t_{i-1}, U^{i-1}) \Delta W^{i, K_{\ell}}\|_H^2 \big]\\
		&= \E\Big[ \Big\| \sum_{k \leq K_{\ell}} \int_{t_{i-1}}^{t_{i}} B_{\ell}(t_{i-1}, U^{i-1}) \big(q_{k}^{\frac{1}{2}} \psi_{k} \big) \diff{\beta_{k}(t)} \Big\|_H^2 \Big]\\
		&= \sum_{k \leq K_{\ell}} \E\Big[ \Big\| \int_{t_{i-1}}^{t_{i}} B_{\ell}(t_{i-1}, U^{i-1}) \big(q_{k}^{\frac{1}{2}} \psi_{k} \big) \diff{\beta_{k}(t)} \Big\|_H^2 \Big]\\
		&= \sum_{k \leq K_{\ell}} \int_{t_{i-1}}^{t_{i}} \E\big[ \big\| B_{\ell}(t_{i-1}, U^{i-1}) \big(q_{k}^{\frac{1}{2}} \psi_{k}\big) \big\|_H^2 \big] \diff{t} \\
		&=  h_i \sum_{k \leq K_{\ell}} \E\big[ \big\| B_{\ell}(t_{i-1}, U^{i-1}) \big(q_{k}^{\frac{1}{2}} \psi_{k}\big) \big\|_H^2 \big]
		\leq h_i \E\big[ \big\| B_{\ell}(t_{i-1}, U^{i-1}) \big\|_{L_2^0}^2 \big],
	\end{align*}
	where in the last step we used the fact that $\{q_{k}^{\frac{1}{2}} \psi_{k}\}_{k \in \N}$ is an orthonormal basis of $Q^{\frac{1}{2}}(H)$. With this in mind, we can find a suitable bound for the expectation of $\Gamma_{4,1}$. We apply the boundedness of $B_{\ell}(t_{i-1}, 0)$ and the Lipschitz continuity of $B_{\ell}(t_{i-1}, \cdot)$ in $L_2^0$ from Assumption~\ref{ass:B} as well as the fact that   
	\begin{align*}
		\| U^{i-1} \|_H^2 
		= \| U_{0}^{i} \|_H^2
		\leq \| U_{0}^{i} \|_H^2 + \| U_{1}^{i} \|_H^2 + \dots + \| U_{s}^{i} \|_H^2
		= \sum_{\ell=1}^{s} \| U_{\ell-1}^{i} \|_H^2,
	\end{align*}
	to find
	\begin{align*}
		\E\big[\Gamma_{4,1}\big]
		&\leq h_i 2 s^2 \sum_{\ell=1}^{s} \E\big[ \big\| B_{\ell}(t_{i-1}, U^{i-1}) \big\|_{L_2^0}^2 \big]
		\leq h_i 4 L_B^2 s^3 \E\big[ \| U^{i-1} \|_H^2 \big]
		+ h_i 4 s^3 K_B^2  \\
		&\leq h_i 4 L_B^2 s^3 \sum_{\ell=1}^{s} \E\big[ \| U_{\ell-1}^{i} \|_H^2 \big]
		+ h_i 4 K_B^2 s^3,  
	\end{align*}
	where we abbreviate $C_{\Gamma_4,1} := 4 L_B^2 s^3$ and $C_{\Gamma_4,2} := 4 K_B^2 s^3$.
	Next we bound the expectation of $\Gamma_{4 ,2}$. This can be done by inserting a telescopic sum, which we then bound by an application of H\"older's inequality for sums. Then it follows that
	\begin{align*}
		\E \big[\Gamma_{4 ,2} \big] 
		&= \frac{1}{8 s^2} \sum_{\ell=1}^{s} \E \Big[ \Big\| \sum_{j = 1}^{\ell} \big(U_{j}^i - U_{j-1}^{i}\big) \Big\|_H^2\Big]\leq \frac{1}{8 	s^2} \sum_{\ell=1}^{s} \E \Big[ \ell \sum_{j = 1}^{\ell} \big\| U_{j}^i - U_{j-1}^{i}\big\|_H^2\Big]\\
		&\leq \frac{1}{8 s^2} \sum_{\ell=1}^{s} \E \Big[ s \sum_{j=1}^{s} \big\| U_{j}^i - U_{j-1}^{i}\big\|_H^2\Big]
		= \frac{1}{8} \sum_{j=1}^{s} \E \big[ \big\| U_{j}^i - U_{j-1}^{i} \big \|_H^2 \big].
	\end{align*}
	Due to the stochastic independence of $B(t_{i-1}, U^{i-1})$ and $\Delta W^{i, K_{\ell}}$, we observe that the last term $\Gamma_{4,3}$ disappears in expectation. Thus, combining all the previous bounds, it follows that
	\begin{align*}
		&\frac{1}{2} \Big( \E\big[\|U^i\|_H^2\big] - \E\big[\|U^{i-1}\|_H^2\big] + \sum_{\ell=1}^{s} \E\big[\| U_{\ell}^i - U_{\ell-1}^i\|_H^2\big]\Big)\\
		&\leq
		h_i \big( C_{\Gamma_1,1} + C_{\Gamma_4,1} \big) \sum_{\ell=1}^{s} \|U_{\ell-1}^i \|_H^2 
		+ h_i \big( C_{\Gamma_1,2} + C_{\Gamma_2} + C_{\Gamma_4,2}\big)\\
		&- h_i \frac{\mu}{2} \sum_{\ell=1}^{s} \|U_{\ell}^i\|_{V_{\ell}}^2
		+ \frac{1}{4} \sum_{j=1}^{s} \E \big[ \big\| U_{j}^i - U_{j-1}^{i} \big \|_H^2 \big].
	\end{align*}
	Rearranging the terms and multiplying the result by $2$, shows that
	\begin{align*}
		&\E\big[\|U^i\|_H^2\big] - \E\big[\|U^{i-1}\|_H^2\big] + \frac{1}{2} \sum_{\ell=1}^{s} \E\big[\| U_{\ell}^i - U_{\ell-1}^i\|_H^2\big]
		+ h_i \mu \sum_{\ell=1}^{s} \E\big[\|U_{\ell}^i\|_{V_{\ell}}^2\big]\\
		&\leq h_i 2 \big( C_{\Gamma_1,1} + C_{\Gamma_4,1} \big) \sum_{\ell=1}^{s} \E \big[\|U_{\ell-1}^i \|_H^2 \big]
		+ h_i 2 \big( C_{\Gamma_1,2} + C_{\Gamma_2} + C_{\Gamma_4,2}\big).
	\end{align*}
	We sum up the inequality from $i = 1$ to $n$ and use a telescoping sum argument to obtain
	\begin{align*}
		&\E\big[\|U^n\|_H^2\big] - \E\big[\|U^0\|_H^2\big] + \frac{1}{2} \sum_{i = 1}^{n} \sum_{\ell=1}^{s} \E\big[\| U_{\ell}^i - 	U_{\ell-1}^i\|_H^2\big]
		+ \mu \sum_{i = 1}^{n} h_i \sum_{\ell=1}^{s} \E\big[\|U_{\ell}^i\|_{V_{\ell}}^2\big]\\
		&\leq  2 \big( C_{\Gamma_1,1} + C_{\Gamma_4,1} \big) \sum_{i=1}^{n} h_i \sum_{\ell=1}^{s} \E \big[\|U_{\ell-1}^i \|_H^2 \big]
		+ 2 T \big( C_{\Gamma_1,2} + C_{\Gamma_2} + C_{\Gamma_4,2}\big).
	\end{align*}
	For the next step, we will apply Gr\"onwall's inequality analogously to the proof of Theorem~\ref{thm:errorBound}. Then we find
	\begin{align*}
		&\E\big[\|U^n\|_H^2\big] + \frac{1}{2} \sum_{i = 1}^{n} \sum_{\ell=1}^{s} \E\big[\| U_{\ell}^i - U_{\ell-1}^i\|_H^2\big]
		+ \mu \sum_{i = 1}^{n} h_i \sum_{\ell=1}^{s} \E\big[\|U_{\ell}^i\|_{V_{\ell}}^2\big]\\
		&\leq \exp \big(2 s T \big( C_{\Gamma_1,1} + C_{\Gamma_4,1} \big)\big)
		\big(\E\big[\|U^0\|_H^2\big] + 2 T \big( C_{\Gamma_1,2} + C_{\Gamma_2} + C_{\Gamma_4,2}\big) \big),
	\end{align*}
	which is finite.
\end{proof}

\end{document}